\documentclass[12pt]{amsart}
\usepackage{fullpage}
\usepackage{amsmath,amssymb,amsthm,mathtools}
\usepackage{booktabs,array,tabularx,longtable,enumitem}
\usepackage[hidelinks]{hyperref}

\hypersetup{
 pdftitle={Density of repelling periodic points for cubic rational maps},
 pdfauthor={Nha Xuan Truong},
 pdfsubject={Cubic rational maps in non-Archimedean dynamics},
 pdfkeywords={non-Archimedean dynamics, cubic rational maps, classical Julia set, type I repelling periodic points, Berkovich projective line}
}

\numberwithin{equation}{section}

\newtheorem{theorem}{Theorem}[section]
\newtheorem{lemma}[theorem]{Lemma}
\newtheorem{proposition}[theorem]{Proposition}
\theoremstyle{definition}

\newcommand{\Ctwo}{\mathbf C_2}
\newcommand{\Cthree}{\mathbf C_3}
\newcommand{\Cpadic}{\mathbf C_p}

\newcommand{\PP}{\mathbf P^1}

\newcommand{\JI}{J_{\mathrm I}}

\newcommand{\DD}{\mathbb D}

\newcommand{\Dbar}{\overline{\mathbb D}}

\newcommand{\OO}{\mathcal O}
\newcommand{\mm}{\mathfrak m}
\newcommand{\zG}{\zeta_{\mathrm G}}

\newcommand{\Rep}{\operatorname{Rep}_{\mathrm I}}

\title[Repelling cycles of cubic rational maps]
{Density of repelling periodic points for cubic rational maps}
\author{Nha Xuan Truong}
\address{Beijing International Center for Mathematical Research,
Peking University, Beijing, China}
\email{nxtruong@bicmr.pku.edu.cn}

\begin{document}

\begin{abstract}
We prove the repelling-density conjecture for cubic rational
maps over $\mathbf C_p$: for every prime $p$, the classical
Julia set is the closure of the classical repelling periodic
points. By B\'ezivin's theorem, it suffices to prove that a
nonempty classical Julia set contains a classical repelling
periodic point. The case $p\geq5$ follows from known results
for tame maps, so we treat residue characteristics two and
three. We argue by contraposition, assuming that the map
has no classical repelling periodic points. Maps with
potential good reduction already have empty classical
Julia set. Otherwise, we conjugate the map to an explicit
normalized form and analyze its coefficients. In each case, we obtain coefficient restrictions and
establish a uniform bound on the spherical derivatives
of all iterates. This bound implies equicontinuity on the
classical projective line and hence emptiness of the
classical Julia set.
\end{abstract}
\maketitle

\section{Introduction}
\label{sec:main}

Let $p$ be a prime, and let $\Cpadic$ be the completion
of an algebraic closure of $\mathbf Q_p$. For a rational
map $f\in\Cpadic(z)$ of degree at least two, the
\emph{classical Fatou set} consists of the points of
$\PP(\Cpadic)$ having a neighborhood on which the
iterates of $f$ are equicontinuous for the chordal
metric. Its complement is the \emph{classical Julia
set}, denoted by $\JI(f)$. Equivalently, $\JI(f)$ is
the intersection of the Berkovich Julia set with
$\PP(\Cpadic)$.

A classical periodic point of least period $n$ is
\emph{repelling} if the multiplier of $f^n$ at that
point has absolute value greater than one.
We write $\Rep(f)$ for the set of classical
repelling periodic points of all periods.
Throughout the paper, ``classical'' and ``type~I''
have the same meaning.

We prove the repelling-density conjecture for cubic rational
maps over $\Cpadic$ for every prime $p$.
The conjecture, discussed in the survey of Benedetto, Ingram, Jones, Manes, Silverman, and Tucker
\cite{survey19}, asserts that the classical Julia set is
the closure of the classical repelling periodic points.

Repelling periodic points are known to be dense in the
Berkovich Julia set. The corresponding question for
classical points is more delicate: this density does
not ensure that classical Julia points can be
approximated by classical repelling periodic points
\cite{survey19}.
\subsection{The main theorem and background}

\begin{theorem}
\label{thm:main-global}
Let $p$ be a prime, and let $f\in\Cpadic(z)$
have degree three. Then
\[
 \JI(f)=\overline{\Rep(f)},
\]
where the closure is taken in $\PP(\Cpadic)$
for the chordal topology. In particular,
\[
 \JI(f)\ne\varnothing
 \quad\Longleftrightarrow\quad
 \Rep(f)\ne\varnothing.
\]
\end{theorem}

The question goes back to Hsia's work on weak
N\'eron models and non-Archimedean dynamics
\cite{hsia96}.
Hsia subsequently proved that the classical
Julia set is contained in the closure of the
classical periodic points and formulated the
repelling-density conjecture
\cite{hsia00}. Approximation
by periodic points alone does not establish
repelling density, since a non-Archimedean
rational map may have infinitely many
nonrepelling periodic points.

B\'ezivin proved that the existence of one
classical repelling periodic point implies
that the classical Julia set is the closure
of the classical repelling periodic points.
He also proved that every classical Julia
point is an accumulation point of such
points \cite{bezivin}.
Since classical repelling periodic points
belong to the classical Julia set, the
conjecture over $\Cpadic$ is equivalent
to the existence implication
\[
 \JI(f)\ne\varnothing
 \quad\Longrightarrow\quad
 \Rep(f)\ne\varnothing.
\]
Our main goal is to prove this implication
for cubic rational maps.

Several cases were already known.
Doyle--Jacobs--Rumely proved the conjecture
for quadratic rational maps
\cite{djr16}.
For cubic polynomials, Briend--Hsia proved
an emptiness criterion in terms of
nonrepelling fixed points and obtained
the resulting density statement over
$\Cpadic$
\cite{briend-hsia12}.
For rational maps, one must also control
orbits passing near poles.

Favre--Rivera-Letelier proved the existence
implication when the residue characteristic
exceeds the degree
\cite{frl26}.
Thus the cubic case for $p\geq5$ is already
known. The arguments developed here address
the remaining residue characteristics,
$p=2$ and $p=3$.

\subsection{Outline of the proof}

We argue by contraposition: assuming that
$f$ has no classical repelling periodic
point, we prove that $\JI(f)=\varnothing$.
Maps with potential good reduction already
have empty classical Julia set
\cite{benedetto},
so we restrict to persistent bad reduction.

Lemma~\ref{nm:lem:fixed}, based on Rivera--Letelier's
fixed-point theorem \cite{rivera}, gives a fixed
type~II point of local degree two.
The absence of classical repelling periodic
points constrains the reduction at this
point and allows us to conjugate $f$ to
\[
 \begin{gathered}
 H(z)=\frac{Az^3+Bz^2+Cz+D}{z(1+Ez+Fz^2)},\\
 |A|=1,\qquad |B|,|C|,|D|<1,\qquad D\ne0,
 \qquad |E|,|F|\leq1.
 \end{gathered}
\]

We analyze this expression using two analytic tools.
First, we construct restrictions of iterates
$H^N:U\to\DD(0,R)$ that are finite and surjective.
When $\sup_{z\in U}|z|<R$, Lemma~\ref{tk:lem:count}
gives a fixed point $h$ of $H^N$ in $U$.
If $|(H^N)'(h)|>1$, then $h$ is a repelling periodic point
of $H$, contrary to our assumption.
The resulting coefficient restrictions give a bound on the
spherical derivatives of all iterates that is independent of
the starting point and the iterate. We obtain this bound either
by multiplying the derivative norms along orbit segments
between successive visits to a specified circle, or
by constructing a function $\phi$ and constants
$0<m\leq M<\infty$ satisfying
\[
 m\leq\phi(z)\leq M,\qquad
 H^\#(z)\phi(z)\leq\phi(H(z))
 \qquad\text{for all }z\in\PP(\Cpadic).
\]
In the latter case, multiplying the inequalities along an
orbit cancels the intermediate values of $\phi$.
Lemma~\ref{tk:lem:fatou} converts this bound into
equicontinuity and hence an empty classical Julia set. This approach
is related to the idea proposed by Thomas Silverman for polynomials
\cite{tsilverman20}.

The paper is organized as follows.
Section~\ref{sec:tame} records the known
case $p\geq5$. Section~\ref{sec:foundations}
establishes the conjugation and analytic
tools. Sections~\ref{sec:two-adic}
and~\ref{sec:three-adic} treat residue
characteristics two and three.
We conclude by combining these results with
B\'ezivin's theorem to prove
Theorem~\ref{thm:main-global}.

\subsection*{Notation and conventions}

We use the absolute value satisfying
$|p|=p^{-1}$ and write
\[
 v_p(x)=-\log_p|x|\quad(x\ne0),
 \qquad v_p(0)=\infty.
\]
We use $f^n$ for the $n$th iterate,
$\OO$ for the valuation ring of $\Cpadic$,
$\mm$ for its maximal ideal, and
$k=\OO/\mm$ for the residue field.
Bars denote coefficientwise reduction
modulo $\mm$. The reduced rational map,
obtained after cancelling common factors,
is denoted by $\widetilde f$.

For finite points, the chordal distance is
\[
 \Delta(z,w)=
 \frac{|z-w|}
 {\max\{1,|z|\}\max\{1,|w|\}},
\]
with
\[
 \Delta(z,\infty)=\frac1{\max\{1,|z|\}},
 \qquad
 \Delta(\infty,\infty)=0.
\]
All closures of sets of classical points
are taken in the chordal topology on
$\PP(\Cpadic)$.
We write
\[
 \DD(a,r)=\{z\in\Cpadic:|z-a|<r\},
 \qquad
 \Dbar(a,r)=\{z\in\Cpadic:|z-a|\leq r\}.
\]

Throughout the paper, roots of polynomial equations are counted
with multiplicity. In particular, two roots listed separately
need not be distinct.

For the finite analytic maps between disks
used below, the mapping degree is the
number of solutions of $f(z)=y$ in the source disk,
for any point $y$ in the target disk.
It is computed using the Weierstrass degree
and refers to the restriction to the disk,
not the global degree of the rational map.

An analytic map of mapping degree one between the
disks considered below is bijective and scales distances
by a constant factor equal to its derivative norm.
These facts follow from the power-series description
of disk images \cite{benedetto}.

\subsection*{Acknowledgments}

We are grateful to Robert Benedetto for bringing this
question to our attention, for telling us about Thomas
Silverman's approach, and for sharing his insights into
non-Archimedean dynamics. We also thank him for his helpful
correspondence and encouragement.
\section{The known tame cubic case}
\label{sec:tame}

We first treat $p\geq5$. Since the residue
characteristic exceeds the degree, every cubic
rational map over $\Cpadic$ is tame
\cite{frl26}.
The existence implication needed for
Theorem~\ref{thm:main-global} follows from
the following stronger result.

\begin{proposition}[Favre]
\label{prop:tame}
Let $p\geq5$, and let $f\in\Cpadic(z)$
have degree three. If $\JI(f)\ne\varnothing$,
then $f$ has a classical repelling periodic
point of period at most three.
\end{proposition}

\begin{proof}
A map with potential good reduction has
empty classical Julia set
\cite{benedetto}.
Thus $\JI(f)\ne\varnothing$ implies that
$f$ has persistent bad reduction.

Since $p>3$, Favre's cubic theorem and its
stated extension to tame cubics apply
\cite[Theorem~3.2 and the following paragraph]{favre}.
They give a rigid repelling cycle of period
at most three. Over $\Cpadic$, rigid points
are classical points, so this proves the
assertion.
\end{proof}

For existence alone, one may instead apply
Favre--Rivera-Letelier's theorem
\cite{frl26}, which gives a
classical repelling periodic point whenever
the classical Julia set of a tame map is
nonempty. This conclusion is sufficient
for our main theorem: B\'ezivin's theorem
\cite{bezivin} then gives
\[
 \JI(f)=\overline{\Rep(f)}
\]
and shows that every classical Julia point
is an accumulation point of classical
repelling periodic points. If $\JI(f)$ is
empty, the same equality holds because
every classical repelling periodic point
belongs to $\JI(f)$.

It remains to treat $p=2$ and $p=3$, which we do in
Sections~\ref{sec:two-adic} and~\ref{sec:three-adic}.
\section{Normalization, the derivative criterion, and coefficient restrictions}
\label{sec:foundations}

We argue by contraposition: if a cubic rational map
$f$ has no classical repelling periodic point, we
prove that $\JI(f)=\varnothing$.

The argument has three parts. First, if $f$ does not have
potential good reduction, we conjugate it to an
explicit map $H$. Second, we give
a criterion for the classical Julia set to be empty: a uniform
bound on the spherical derivatives of all iterates implies
equicontinuity. We obtain this bound either directly or through
a function bounded above and bounded away from zero satisfying
a one-step derivative inequality.
We also record the lemmas on return maps between disks used to obtain
coefficient restrictions. Third, we establish the
initial restrictions on the normalized map before
separating the two-adic and three-adic arguments.

\subsection{Replacing the map by a normalized conjugate}
\label{nm:sec:normalization}

For $\phi\in\operatorname{PGL}_2(\Cpadic)$, write
\[
 f^\phi=\phi^{-1}\circ f\circ\phi.
\]
We say that $f$ has \emph{potential good reduction}
if some conjugate $f^\phi$ has good reduction.
Otherwise, $f$ has \emph{persistent bad reduction}.
Conjugation preserves periods and multipliers
of classical periodic points, as well as emptiness
of the classical Julia set.

A map with potential good reduction has empty
classical Julia set
\cite{benedetto}.
Thus only persistent bad reduction requires further
consideration. We first find a fixed type~II point
and study the reduction there.

\begin{lemma}
\label{nm:lem:fixed}
Let $f\in\Cpadic(z)$ have degree three and persistent
bad reduction. If $f$ has no classical repelling
periodic point, then it has a fixed type~II point
of local degree two.
\end{lemma}

\begin{proof}
By Rivera--Letelier's fixed-point theorem
\cite[Theorem~B]{rivera}, $f$ has a repelling fixed
point $\zeta$ on the Berkovich projective line.
The hypothesis excludes classical repelling fixed
points, so $\zeta$ is nonclassical.

A nonclassical repelling fixed point satisfies
$\deg_\zeta(f)\geq2$
\cite{benedetto}.
Periodic points of types~III and IV have local
degree one \cite{benedetto}.
Hence $\zeta$ is of type~II, with local degree
either two or three.

Conjugate $\zeta$ to the Gauss point. Its local
degree is the degree of the reduced map
\cite{benedetto}.
If this degree were three, the conjugate would
have good reduction, contradicting persistent
bad reduction. Therefore the local degree is two.
\end{proof}

After conjugation, assume that the point supplied
by Lemma~\ref{nm:lem:fixed} is $\zG$, and continue
to denote the map by $f$. Write
\[
 f([X:Y])=[F(X,Y):G(X,Y)],
\]
where $F,G\in\OO[X,Y]$ are coprime homogeneous
cubic forms whose largest coefficient norm is one.
Let $\bar F,\bar G\in k[X,Y]$ be their reductions
modulo $\mm$, where $k=\OO/\mm$.

Cancelling their common factor gives a map of
degree two. Their greatest common divisor therefore
has degree one, so
\begin{equation}
\label{nm:eq:reduction}
 \bar F=LU,\qquad \bar G=LV,
\end{equation}
where $L$ is linear and $U,V$ are coprime homogeneous
quadratic forms over $k$. The reduced map is
$\widetilde f=[U:V]$.

The zeros of the common factor are called the
\emph{holes} of the reduction, with multiplicities
given by their orders of vanishing in that factor.
Thus the zero $h\in\PP(k)$ of $L$ is the unique
simple hole. We show that it is totally invariant
under $\widetilde f$.

For $a\in k$, choose a lift $\alpha\in\OO$ and put
\[
 D_a=\DD(\alpha,1).
\]
This disk is independent of the chosen lift.
Also put
\[
 D_\infty=\{z\in\Cpadic:|z|>1\}\cup\{\infty\}.
\]
These residue classes are the classical points
in the directions at the Gauss point indexed
by $\PP(k)$.

\begin{lemma}
\label{nm:lem:hole}
Let $g\in\Cpadic(z)$ fix the Gauss point, and let
$\widetilde g\in k(z)$ be its nonconstant reduced map.
Suppose that the reduction has a simple hole
at $a\in\PP(k)$ and that $\widetilde g(a)\ne a$.
Then $D_a$ contains a classical repelling fixed
point of $g$.
\end{lemma}

\begin{proof}
Conjugate by an element of
$\operatorname{PGL}_2(\OO)$ so that
\[
 a=0,\qquad \widetilde g(0)\in k^\times.
\]
Such a coordinate change can be chosen over $k$
and lifted to $\OO$. It preserves the Gauss point
and sends $D_a$ to $\DD(0,1)$.

Write $g=P/Q$, with coprime $P,Q\in\OO[z]$
whose largest coefficient norm is one.
Since zero is a simple hole and $\widetilde g(0)$
is finite and nonzero,
\[
 \bar P(z)=z\,u(z),\qquad
 \bar Q(z)=z\,v(z),
\]
where $u(0)v(0)\ne0$.

Consider $T(z)=P(z)-zQ(z)$. Its reduction satisfies
\[
 \bar T(z)=z\,u(z)-z^2v(z),\qquad
 \bar T(0)=0,\qquad \bar T'(0)=u(0)\ne0.
\]
Hensel's lemma gives a root $x\in\mm$ of $T$
with $|T'(x)|=1$. Coprimality implies $Q(x)\ne0$,
so $g(x)=x$. Also $0<|Q(x)|<1$.

Differentiating $g(z)-z=T(z)/Q(z)$ at $x$ gives
\[
 g'(x)-1=\frac{T'(x)}{Q(x)},\qquad
 |g'(x)-1|=\frac1{|Q(x)|}>1.
\]
The ultrametric inequality implies $|g'(x)|>1$.
Conjugating back gives the required repelling
fixed point in $D_a$.
\end{proof}

\begin{proposition}
\label{nm:prop:total}
In the setting of \eqref{nm:eq:reduction}, suppose
that $f$ has no classical repelling periodic point.
Then the zero $h$ of $L$ is totally invariant:
\[
 \widetilde f(h)=h,\qquad
 \widetilde f^{-1}(\{h\})=\{h\}.
\]
\end{proposition}

\begin{proof}
Lemma~\ref{nm:lem:hole} first gives
$\widetilde f(h)=h$.
Suppose that $a\ne h$ also satisfies
$\widetilde f(a)=h$. Since the reduced map has
degree two, these two preimages each have
multiplicity one. In particular, $a$ is a simple
zero of $L(U,V)$.

Composing the homogeneous lift with itself gives
\begin{equation}
\label{nm:eq:iteratehole}
 \begin{aligned}
 \overline{F(F,G)}
   &=L^3L(U,V)\,U(U,V),\\
 \overline{G(F,G)}
   &=L^3L(U,V)\,V(U,V).
 \end{aligned}
\end{equation}
These forms are not both zero, so the composed
lift has largest coefficient norm one.
The forms $U(U,V)$ and $V(U,V)$ are coprime:
a common projective zero would map under
$\widetilde f$ to a common zero of $U$ and $V$.
Thus their full common factor in
\eqref{nm:eq:iteratehole} is $L^3L(U,V)$,
and the reduced map of $f^2$ is $(\widetilde f)^2$.

At $a$, the factor $L$ does not vanish, whereas
$L(U,V)$ vanishes simply. Hence $a$ is a simple
hole of $f^2$, and
\[
 (\widetilde f)^2(a)=\widetilde f(h)=h\ne a.
\]
Lemma~\ref{nm:lem:hole} gives a point $x$ satisfying
$f^2(x)=x$ and $|(f^2)'(x)|>1$.
Here derivatives are computed in local coordinates along the orbit.
If $f(x)=x$, then $(f^2)'(x)=f'(x)^2$, so $|f'(x)|>1$.
Otherwise, $x$ has exact period two and
$|(f^2)'(x)|=|f'(f(x))f'(x)|>1$.
In either case, $x$ is a classical repelling periodic
point of $f$, a contradiction.
This proves total invariance.
\end{proof}

\begin{theorem}
\label{nm:thm:normalization}
Let $f\in\Cpadic(z)$ have degree three and
persistent bad reduction. Suppose that $f$
has no classical repelling periodic point.
Then a conjugate of $f$ has the form
\begin{equation}
\label{nm:eq:onepole}
 H(z)=\frac{Az^3+Bz^2+Cz+D}{z(1+Ez+Fz^2)},
\end{equation}
where
\begin{equation}
\label{nm:eq:norms}
 |A|=1,\qquad |B|,|C|,|D|<1,\qquad D\ne0,
 \qquad |E|,|F|\leq1.
\end{equation}
The Gauss point is fixed by $H$ with local
degree two, and the unique simple hole of
its reduction is totally invariant under
$\widetilde H$.
\end{theorem}
\begin{proof}
By Lemma~\ref{nm:lem:fixed}, $f$ has a fixed type~II
point of local degree two. Conjugate this point to $\zG$.
Proposition~\ref{nm:prop:total} shows that the unique
simple hole of the reduction is totally invariant.

Conjugate further by an element of
$\operatorname{PGL}_2(\OO)$ whose reduction sends the hole
to infinity. Total invariance means that infinity is the
only preimage of infinity under the reduced map.
Thus the reduced map is a quadratic polynomial.

Multiply both homogeneous forms by the same scalar
to normalize the coefficient of $Y^3$ in the denominator
to one. Their reductions are
\[
 \bar F=Y(u_2X^2+u_1XY+u_0Y^2),\qquad
 \bar G=Y^3,
\]
where $u_2\ne0$. Write the homogeneous forms as
\[
 F=a_3X^3+a_2X^2Y+a_1XY^2+a_0Y^3,\qquad
 G=b_3X^3+b_2X^2Y+b_1XY^2+Y^3.
\]
Comparing coefficients in their reductions gives
\[
 \bar a_3=0,\qquad
 \bar a_2=u_2,\qquad
 \bar a_1=u_1,\qquad
 \bar a_0=u_0,\qquad
 \bar b_1=\bar b_2=\bar b_3=0.
\]
Setting $Y=1$, we obtain
\[
 f_0(z)=
 \frac{a_3z^3+a_2z^2+a_1z+a_0}
      {1+b_1z+b_2z^2+b_3z^3},
\]
with
\[
 |a_2|=1,\qquad |a_0|,|a_1|\leq1,\qquad
 |a_3|,|b_1|,|b_2|,|b_3|<1.
\]

Conjugating by $\iota(w)=1/w$ gives
$g(w)=n(w)/d(w)$, where
\[
 n(w)=w^3+b_1w^2+b_2w+b_3,\qquad
 d(w)=a_0w^3+a_1w^2+a_2w+a_3.
\]
The reduction of $d$ has a simple zero at zero.
By Hensel's lemma, $d$ has a unique zero $\eta\in\mm$.
Put
\[
 \alpha=d'(\eta)
       =a_2+2a_1\eta+3a_0\eta^2.
\]
Then $|\alpha|=1$. Since $n$ and $d$ are coprime,
$n(\eta)\ne0$, so $\eta$ is a simple pole of $g$.

Conjugating by $\tau(z)=z+\eta$ moves this pole
to zero and gives
\[
 H(z)=\frac{n(z+\eta)-\eta d(z+\eta)}{d(z+\eta)}.
\]
Since $d(\eta)=0$,
\[
 d(z+\eta)
 =z\bigl(\alpha+(3a_0\eta+a_1)z+a_0z^2\bigr).
\]
Dividing the numerator and denominator by $\alpha$
gives \eqref{nm:eq:onepole}, with
\begin{align*}
 A&=\frac{1-\eta a_0}{\alpha},&
 B&=\frac{3\eta+b_1-\eta(3a_0\eta+a_1)}{\alpha},\\
 C&=\frac{3\eta^2+2b_1\eta+b_2-\eta\alpha}{\alpha},&
 D&=\frac{n(\eta)}{\alpha},\\
 E&=\frac{3a_0\eta+a_1}{\alpha},&
 F&=\frac{a_0}{\alpha}.
\end{align*}
These formulas give $|A|=1$, $|B|,|C|<1$,
and $|E|,|F|\leq1$. Also $0<|n(\eta)|<1$,
so \eqref{nm:eq:norms} holds.

All coordinate changes after moving the type~II
point to $\zG$ belong to $\operatorname{PGL}_2(\OO)$.
They preserve the Gauss point, its local degree,
and total invariance of the hole.
\end{proof}

\subsection{Proving that the classical Julia set is empty}
\label{tk:sec:tools}

Our goal is to bound $(f^n)^\#(z)$ independently of $z$ and $n$.
Lemma~\ref{tk:lem:fatou} shows that such a bound implies
equicontinuity and hence that the classical Julia set is empty.
In some cases, we estimate derivative products directly.
In others, it is simpler to construct a function $\phi$ and
constants $0<m\leq M<\infty$ such that
\[
 m\leq\phi(z)\leq M,\qquad
 f^\#(z)\phi(z)\leq\phi(f(z))
 \qquad\text{for all }z\in\PP(\Cpadic).
\]
Multiplying these inequalities along an orbit
makes the intermediate factors cancel.
This gives the required uniform bound on the spherical
derivatives of all iterates.

Before proving this criterion, we record two
lemmas on return maps between disks. They will be used to obtain
coefficient restrictions: a fixed point $x$ of an iterate
$T=f^N$ with $|T'(x)|>1$ would be a classical repelling
periodic point of $f$.

Indeed, a fixed point $x$ of $T=f^N$ is periodic
under $f$. If its exact period is $r$, then
$r\mid N$ and
\[
 |(f^N)'(x)|=|(f^r)'(x)|^{N/r}.
\]
Thus $|T'(x)|>1$ shows that the periodic point is repelling,
without requiring its exact period.

\begin{lemma}
\label{tk:lem:count}
Let $U$ be an open disk, and let
$T:U\to\DD(0,R)$ be a finite surjective
analytic map. Suppose that
\[
 \sup_{z\in U}|z|<R.
\]
Then $T$ has a fixed point in $U$.
\end{lemma}

\begin{proof}
Put $c=\sup_{z\in U}|z|$ and choose $y$ with
$c<|y|<R$. By surjectivity, choose $a,b\in U$
with $T(a)=0$ and $T(b)=y$. The closed disk $V=\Dbar(a,|b-a|)$
is contained in $U$. The supremum norm of $T$
on $V$ is at least $|y|$, whereas that of the
function $z$ is at most $c$.

Subtracting $z$ therefore leaves the Weierstrass
degree unchanged. This degree is positive because
$T(a)=0$ and $T$ is nonconstant. The analytic
zero-counting theorem gives a zero of $T(z)-z$
in $V$; see
\cite{benedetto}.
\end{proof}

We sometimes replace an intermediate disk by a smaller disk
when constructing a composition of disk maps.
The next lemma describes the disks in its inverse image.
Mapping degree refers to the restriction to each disk,
as defined in the notation section.

\begin{lemma}
\label{tk:lem:refine}
Let $U$ be an open disk, and let $f=P/Q$, where
$P,Q\in\Cpadic[z]$ are coprime. Suppose that
$f:U\to\DD(b,r)$ is surjective of mapping degree
two and that
\[
 |Q(z)|=q>0
 \qquad\text{for all }z\in U.
\]
For $0<s\leq r$, the inverse image of $\DD(b,s)$
inside $U$ consists of either one open disk
mapping onto $\DD(b,s)$ with degree two, or
two disjoint open disks each mapping onto
$\DD(b,s)$ with degree one.
\end{lemma}

\begin{proof}
For every $z\in U$,
\[
 |f(z)-b|<s
 \quad\Longleftrightarrow\quad
 |P(z)-bQ(z)|<sq.
\]
Factoring $P-bQ$ describes this set as a finite
disjoint union of open disks, each containing
a root of $P-bQ$ in $U$.

Let $V$ be one of these disks and let $y\in\DD(b,s)$.
Clearing denominators in the equation $f(z)=y$ gives
\[
 P(z)-yQ(z)
 =P(z)-bQ(z)-(y-b)Q(z)=0.
\]
The term $(y-b)Q(z)$ has norm $|y-b|q<sq$
throughout $V$.
Choose sufficiently large closed subdisks of $V$
containing all roots of $P-bQ$ in $V$, on which
the supremum norm of $P-bQ$ exceeds $|y-b|q$.
Subtracting $(y-b)Q$ leaves their Weierstrass
degrees unchanged. Taking increasingly large closed disks
contained in $V$ shows that $P-yQ$ and $P-bQ$ have the same
number of roots in $V$.

Thus each of these disks maps onto $\DD(b,s)$ with positive
mapping degree. The sum of their mapping degrees is two,
so the stated alternatives follow.
\end{proof}

When we construct a composition of disk maps by taking
inverse images, we choose one disk in the inverse image
at each step. Lemma~\ref{tk:lem:refine} ensures that replacing
an intermediate target disk by a smaller disk still gives
a surjective map on each of the resulting preimage disks.
The composition therefore maps its first disk onto the
prescribed final disk. We apply Lemma~\ref{tk:lem:count}
after verifying
\[
 \sup_{z\in U}|z|<R,
\]
where $U$ is the first disk and $\DD(0,R)$ is the final disk.
We then check that the resulting periodic point is repelling
by multiplying the derivative norms along its orbit.

We now explain how derivative bounds imply that
the classical Julia set is empty. At a finite point where $f$
is finite, its spherical derivative is
\[
 f^\#(z)
 =|f'(z)|
 \frac{\max\{1,|z|\}^2}
      {\max\{1,|f(z)|\}^2}.
\]
At poles and infinity, it is defined using local
coordinates that are chordal isometries.
It satisfies the chain rule
\[
 (f^{n+m})^\#(x)
 =(f^n)^\#(f^m(x))(f^m)^\#(x)
 \qquad\text{for all }x\in\PP(\Cpadic)
 \text{ and }n,m\geq0.
\]
At a periodic point, the spherical derivative of an iterate
fixing that point is the absolute value of the multiplier
of that iterate.

Uniform boundedness is preserved under fixed
conjugation. If $g=\phi^{-1}\circ f\circ\phi$, then
\[
 (g^n)^\#(x)
 =(f^n)^\#(\phi(x))
 \frac{\phi^\#(x)}{\phi^\#(g^n(x))}
 \qquad\text{for all }x\in\PP(\Cpadic)
 \text{ and }n\geq0.
\]
For fixed $\phi\in\operatorname{PGL}_2(\Cpadic)$,
the function $\phi^\#$ is bounded above and
bounded away from zero.
\begin{lemma}
\label{tk:lem:fatou}
Let $f\in\Cpadic(z)$ be nonconstant. Suppose that
there exists $M\geq1$ such that
\[
 (f^n)^\#(x)\leq M
 \qquad\text{for all }x\in\PP(\Cpadic)
 \text{ and }n\geq0.
\]
Then the iterates of $f$ are equicontinuous,
and $\JI(f)=\varnothing$.
\end{lemma}

\begin{proof}
We first obtain constants $\eta>0$ and $C\geq1$,
depending only on $f$, such that
\begin{equation}
\label{tk:eq:localdistance}
 \Delta(f(x),f(y))
 \leq\max\{f^\#(x)\Delta(x,y),C\Delta(x,y)^2\}
\end{equation}
for all $x,y\in\PP(\Cpadic)$ with $\Delta(x,y)<\eta$.

Choose a homogeneous lift $[F:G]$ with integral
coefficients and largest coefficient norm one.
Since $F$ and $G$ have no common projective zero,
their B\'ezout identities give a constant
$0<q\leq1$ such that
\[
 \max\{|F(X,Y)|,|G(X,Y)|\}\geq q
 \qquad\text{whenever }\max\{|X|,|Y|\}=1.
\]

For each $x$, choose coordinate changes in
$\operatorname{PGL}_2(\OO)$ sending $x$ and
$f(x)$, respectively, to zero.
They preserve chordal distances and integrality
of the coefficients. In these coordinates,
the numerator has constant term zero and the
denominator has constant term of norm at least $q$.
Expanding the quotient gives
\[
 c_1u+c_2u^2+\cdots,\qquad
 |c_1|=f^\#(x),\qquad
 |c_j|\leq q^{-j}
 \quad\text{for all }j\geq1.
\]
Since the coordinate changes are chordal isometries,
the source coordinate $u$ of $y$ satisfies
$|u|=\Delta(x,y)$ whenever $\Delta(x,y)<q/2$.
For $|u|<q/2$, the terms of degree at least two
have norm at most $q^{-2}|u|^2$, and the image
lies in the open unit disk. Its absolute value
therefore equals $\Delta(f(x),f(y))$.
Thus \eqref{tk:eq:localdistance} holds with
$\eta=q/2$ and $C=q^{-2}$, uniformly in $x$.

Fix $\varepsilon>0$ and choose
\[
 0<r<\min\{\eta,1/(CM),\varepsilon\}.
\]
Let $x,y\in\PP(\Cpadic)$ satisfy $\Delta(x,y)<r/M$, and put
\[
 d_j=\Delta(f^j(x),f^j(y)),\qquad
 a_j=f^\#(f^j(x)).
\]
By the chain rule and the hypothesis,
\[
 \prod_{\ell=j}^{j+m-1}a_\ell
 =(f^m)^\#(f^j(x))\leq M
 \qquad(j,m\geq0),
\]
where an empty product is one.

Suppose that $d_n\geq r$ for some $n$, and choose the
least such $n$. Since $d_0<r/M\leq r$, we have $n\geq1$.
For $0\leq j<n$, we have $d_j<r<\eta$, so
\eqref{tk:eq:localdistance} gives
\[
 d_{j+1}\leq\max\{a_jd_j,Cd_j^2\}.
\]
Repeated substitution yields
\[
 \begin{aligned}
 d_n
 &\leq\max\left\{
 d_0\prod_{\ell=0}^{n-1}a_\ell,\,
 C\max_{0\leq j<n}
 \left(d_j^2\prod_{\ell=j+1}^{n-1}a_\ell\right)
 \right\}\\
 &\leq\max\{Md_0,MC r^2\}<r,
 \end{aligned}
\]
a contradiction.

Thus $\Delta(x,y)<r/M$ implies
\[
 \Delta(f^n(x),f^n(y))<r<\varepsilon
 \qquad\text{for all }n\geq0.
\]
The iterates are equicontinuous, so $\JI(f)=\varnothing$.
\end{proof}
The following theorem shows how to obtain a uniform derivative bound. 

\begin{theorem}
\label{bd:thm:boundedfactor}
Let $f\in\Cpadic(z)$ be nonconstant. Suppose that
there exist a function
\[
 \phi:\PP(\Cpadic)\longrightarrow\mathbb R_{>0}
\]
and constants $0<m\leq M<\infty$ such that
\[
 m\leq\phi(z)\leq M,\qquad
 f^\#(z)\phi(z)\leq\phi(f(z))
 \qquad\text{for all }z\in\PP(\Cpadic).
\]
Then
\[
 (f^n)^\#(z)
 \leq\frac{\phi(f^n(z))}{\phi(z)}
 \leq\frac{M}{m}
 \qquad\text{for all }z\in\PP(\Cpadic)
 \text{ and }n\geq0.
\]
Consequently, $\JI(f)=\varnothing$.
\end{theorem}

\begin{proof}
For every $z\in\PP(\Cpadic)$ and $n\geq1$,
the chain rule gives
\[
 \begin{aligned}
 (f^n)^\#(z)
 &=\prod_{j=0}^{n-1}f^\#(f^j(z))
 \leq\prod_{j=0}^{n-1}
 \frac{\phi(f^{j+1}(z))}{\phi(f^j(z))}
 =\frac{\phi(f^n(z))}{\phi(z)}
 \leq\frac{M}{m}.
 \end{aligned}
\]
The bound also holds for $n=0$.
Lemma~\ref{tk:lem:fatou} then gives
$\JI(f)=\varnothing$.
\end{proof}

\subsection{The normalized map and its coefficient restrictions}

Write the normalized map of
Theorem~\ref{nm:thm:normalization} as
\begin{equation}
\label{tk:eq:H}
 H(z)=\frac{P(z)}{zQ(z)}
 =\frac{Az^3+Bz^2+Cz+D}{z(1+Ez+Fz^2)},
\end{equation}
where
\begin{equation}
\label{tk:eq:norms}
 |A|=1,\qquad |B|,|C|,|D|<1,\qquad D\ne0,
 \qquad |E|,|F|\leq1.
\end{equation}
Throughout, we assume that $H$ has no classical
repelling periodic point. 

Since $|Q(z)|=1$ for $|z|<1$, zero is the
only pole of $H$ in the open unit disk.
The following lemma describes the disk maps
near this pole.

\begin{lemma}
\label{tk:lem:poledisk}
If $0<|\xi|<1$ and $|D|>|\xi|^2$, then $H$
maps
\[
 \DD(\xi,|\xi|)
 \longrightarrow
 \DD\left(H(\xi),\frac{|D|}{|\xi|}\right)
\]
bijectively, with constant derivative norm
$|D|/|\xi|^2$.
\end{lemma}

\begin{proof}
Write
\[
 H(z)=\frac Dz+
 \frac{Az^2+(B-DF)z+C-DE}{Q(z)}.
\]
The second term has integral power-series
coefficients on $\DD(0,1)$. Thus, in the expansion
$H(\xi+w)-H(\xi)=\sum_{k\geq1}a_k w^k$,
\[
 |a_1|=\frac{|D|}{|\xi|^2},
 \qquad
 |a_k|\leq\frac{|D|}{|\xi|^{k+1}}
 \quad(k\geq2).
\]
For $|w|<|\xi|$, the linear term strictly
dominates all higher terms. The analytic
disk-mapping theorem gives the bijection
and its image radius $|D|/|\xi|$.
For every $z\in\DD(\xi,|\xi|)$, these estimates give
\[
 |H'(z)|=\frac{|D|}{|\xi|^2}.
\]
\end{proof}

\begin{proposition}
\label{tk:prop:coefficientbounds}
\label{bd:prop:linear}
We have:
\begin{enumerate}
 \item $|D|\geq|C|^2$.
 \item $|B|\leq\max\{|C|^{1/2},|D|^{1/3}\}$.
 \item If $p=2$ and $|C|>|D|^{2/3}$, then
 $|B|<|C|^{1/2}$.
\end{enumerate}
\end{proposition}

\begin{proof}
We prove each part by contradiction.

\begin{enumerate}
\item Suppose that $|D|<|C|^2$.
Clearing denominators in the equation $H(z)=z$ gives
\[
 P(z)-z^2Q(z)
 =-Fz^4+(A-E)z^3+(B-1)z^2+Cz+D=0.
\]
Its coefficients of degree at least two have
norm at most one. The Newton polygon segment joining the constant
and linear coefficients therefore gives a fixed point
$x$ with
\[
 |x|=\frac{|D|}{|C|}<|C|<1.
\]
Since $|D|/|x|^2=|C|^2/|D|>1$,
Lemma~\ref{tk:lem:poledisk} gives
\[
 |H'(x)|=\frac{|C|^2}{|D|}>1,
\]
a contradiction.

\item Suppose that
$|B|>\max\{|C|^{1/2},|D|^{1/3}\}$.
We construct two successive disk maps whose
composition has a repelling fixed point.

The Newton polygon of $P$ has a segment of
horizontal length one joining the coefficients
of $z^2$ and $z^3$. It gives a root $\zeta$
with $|\zeta|=|B|$, so $H(\zeta)=0$.
Using $P(\zeta)=0$, we have
\[
 A\zeta+B=-\frac{C}{\zeta}-\frac{D}{\zeta^2},
 \qquad |A\zeta+B|<|B|.
\]
Therefore
\[
 P'(\zeta)=A\zeta^2+2\zeta(A\zeta+B)+C
\]
has norm $|B|^2$, and $|H'(\zeta)|=|B|$.
The coefficients of degree $k\geq2$ in the
expansion of $H$ at $\zeta$ have norm at most
$|B|^{2-k}$. Hence the linear term dominates
on $\DD(\zeta,|B|)$, giving a bijective map
\[
 H:\DD(\zeta,|B|)\longrightarrow\DD(0,|B|^2)
\]
with constant derivative norm $|B|$.

To obtain the first disk, solve $H(z)=\zeta$:
\[
 (A-\zeta F)z^3+(B-\zeta E)z^2
 +(C-\zeta)z+D=0.
\]
Since $|C-\zeta|=|B|$, the Newton polygon segment joining the constant
and linear coefficients gives a root $\xi$
with $|\xi|=|D|/|B|$.
Lemma~\ref{tk:lem:poledisk} applies as
$|B|^2/|D|>1$. We obtain
\[
 \DD(\xi,|D|/|B|)
 \xrightarrow{H}\DD(\zeta,|B|)
 \xrightarrow{H}\DD(0,|B|^2),
\]
with derivative norms $|B|^2/|D|$ and $|B|$,
respectively.

Since $|\xi|=|D|/|B|<|B|^2$, the first disk
lies in the final target disk.
Lemma~\ref{tk:lem:count} gives a fixed point
$h$ of $H^2$ in the first disk, with
\[
 |(H^2)'(h)|=\frac{|B|^3}{|D|}>1,
\]
a contradiction.

\item Suppose that $p=2$ and $|C|>|D|^{2/3}$.
Part~(2) gives $|B|\leq|C|^{1/2}$.
If equality holds, then
\[
 |C|=|B|^2,\qquad |D|<|B|^3.
\]
The Newton polygon of $P$ gives two roots with absolute
value $|B|$.
Choose one, denoted by $\zeta$.
Using $P(\zeta)=0$, we obtain
\[
 P'(\zeta)=2A\zeta^2+B\zeta-\frac{D}{\zeta}.
\]
Since $p=2$, the term $B\zeta$ strictly
dominates. Thus $|P'(\zeta)|=|B|^2$ and
$|H'(\zeta)|=|B|$.

The coefficients of $w^k$ in $H(\zeta+w)$ have norm
at most $|B|^{2-k}$ for $k\geq2$.
Since $|H'(\zeta)|=|B|$, the linear term dominates
on $\DD(\zeta,|B|)$. Hence
\[
 H:\DD(\zeta,|B|)\longrightarrow\DD(0,|B|^2)
\]
is bijective.

The Newton polygon of the equation $H(z)=\zeta$ gives
a solution $\xi$ with $|\xi|=|D|/|B|$.
Lemma~\ref{tk:lem:poledisk} gives
\[
 \DD(\xi,|D|/|B|)
 \xrightarrow{H}\DD(\zeta,|B|)
 \xrightarrow{H}\DD(0,|B|^2).
\]
Since $|D|/|B|<|B|^2$, Lemma~\ref{tk:lem:count}
gives a fixed point $h$ of $H^2$ in $\DD(\xi,|D|/|B|)$
with
\[
 |(H^2)'(h)|
 =\frac{|B|^3}{|D|}
 =\frac{|C|^{3/2}}{|D|}>1.
\]
This contradiction proves the strict inequality.
\end{enumerate}
\end{proof}

Proposition~\ref{tk:prop:coefficientbounds} divides the
remaining analysis into three mutually exclusive cases:
\begin{enumerate}
 \item Case~1: $|C|^2<|D|<|C|^{3/2}$;
 \item Case~2: $|D|\geq|C|^{3/2}$;
 \item Case~3: $|D|=|C|^2$.
\end{enumerate}
These exhaust all possibilities, with $C=0$
belonging to Case~2. In Cases~1 and~3,
$|B|\leq|C|^{1/2}$, strictly when $p=2$.
In Case~2, $|B|\leq|D|^{1/3}$.

For the subsequent case analysis, we record the derivative
estimates used to bound products along orbits and to construct
a function $\phi$ bounded above and bounded away from zero.

\begin{lemma}
\label{tk:lem:derivativeestimates}
For $z\ne0$ away from the poles of $H$,
\begin{equation}
\label{tk:eq:derivative}
 H'(z)=
 \frac{Az(2+Ez)+B(1-Fz^2)-C(E+2Fz)
       -D(z^{-2}+2Ez^{-1}+3F)}{Q(z)^2}.
\end{equation}
In particular,
\begin{equation}
\label{tk:eq:derivativebound}
 |H'(z)|\leq
 \max\{|2|\,|z|,|z|^2,|B|,|C|,|D|/|z|^2\}
 \qquad \textrm{for }0<|z|<1.
\end{equation}
Moreover,
\[
 \mathcal E=\{z\in\Cpadic:|z|\geq1\}\cup\{\infty\}
\]
is forward invariant, and $H^\#(z)\leq1$
for every $z\in\mathcal E$.
\end{lemma}

\begin{proof}
The derivative formula follows by differentiation.
Since $|Q(z)|=1$ for $|z|<1$, the ultrametric
triangle inequality gives the upper bound.

In the coordinate $w=1/z$, the conjugate map is
\[
 \frac1{H(1/w)}
 =\frac{w^2+Ew+F}{A+Bw+Cw^2+Dw^3}.
\]
On $|w|\leq1$, its denominator has norm one,
and the map has integral power-series
coefficients. It therefore preserves the
closed unit disk and has derivative norm
at most one there. Inversion preserves the
chordal metric, giving the assertions on
$\mathcal E$.
\end{proof}

If $\phi$ is constant on $\mathcal E$,
Lemma~\ref{tk:lem:derivativeestimates}
verifies $H^\#(z)\phi(z)\leq\phi(H(z))$
on this region, leaving only the open unit
disk to consider.

\section{The two-adic case}
\label{sec:two-adic}

The goal of this section is to prove a uniform bound
on the spherical derivatives of all iterates.

\begin{theorem}
\label{sy:cor:normalform}
Let $H\in\Ctwo(z)$ satisfy
\eqref{tk:eq:H}--\eqref{tk:eq:norms}.
If $H$ has no classical repelling periodic point, then there
exists a finite constant $M\geq1$, depending only on $H$, such that
\[
 (H^n)^\#(z)\leq M
 \qquad\text{for all }z\in\PP(\Ctwo)
 \text{ and }n\geq0.
\]
Consequently, $\JI(H)=\varnothing$.
\end{theorem}

We establish the bound case by case.
Return maps between disks first give the necessary coefficient
restrictions. In Cases~1 and~3, multiplying the derivative
inequalities along an orbit cancels the intermediate values
of an auxiliary function $\phi$.
In Case~2, we first bound derivative products between
successive visits to the circle $|z|=d$.
Once the uniform bound is proved in all three cases,
Lemma~\ref{tk:lem:fatou} shows that the classical Julia set is empty.

Retain the notation $P,Q,\mathcal E$ from
Section~\ref{sec:foundations}, and put
\[
 \delta(z)=
 \begin{cases}
  \min\{|z|,1\},&z\in\Ctwo,\\
  1,&z=\infty.
 \end{cases}
\]

\subsection
{Case 1: \texorpdfstring{$|C|^2<|D|<|C|^{3/2}$}{|C|\string^2 < |D| < |C|\string^(3/2)}}
\label{sp:sec:plateau}

Put
\[
 c=|C|^{1/2},\qquad \lambda=|D|/|C|.
\]
Then
\[
 |C|=c^2,\qquad |D|=c^2\lambda,\qquad
 c^2<\lambda<c<1,\qquad |B|<c.
\]
\begin{lemma}
\label{sp:lem:coefficientbounds}
We have
\[
 \lambda\geq\frac14,\qquad |B|\leq\lambda.
\]
\end{lemma}
\begin{proof}
The Newton polygon of $P$ gives one zero $\zeta$ with
absolute value $\lambda$ and two zeros with absolute value $c$.

First we prove that $|B|\leq\max\{\lambda,c/2\}$.
Choose one of the zeros $\xi$ of $P$ with $|\xi|=c$
and write $z=\xi u$.
On $|u|=1$, the reduction of $H(\xi u)/C$ is
$\bar u^2+1$, since $H(\xi)=0$.
For every $|y|<c^2$, the equation $H(z)=y$ has two roots
with absolute value $c$, and their reductions satisfy
$\bar u^2+1=0$.
In characteristic two, this implies $\bar u=1$.
Both roots therefore satisfy $|z-\xi|<c$, and $H:\DD(\xi,c)\longrightarrow\DD(0,c^2)$
is surjective of mapping degree two.

The Newton polygon of the equation $H(z)=\xi$ gives
a solution $\eta$ with $|\eta|=c\lambda$.
Lemma~\ref{tk:lem:poledisk} gives
\[
 \DD(\eta,c\lambda)
 \xrightarrow{H}\DD(\xi,c)
 \xrightarrow{H}\DD(0,c^2).
\]
The first map is bijective.
Since every point in the first disk has absolute value
$c\lambda<c^2$, Lemma~\ref{tk:lem:count} gives a point
$h$ in that disk with $H^2(h)=h$.
For $x=H(h)$, we have
\[
 |h|=c\lambda,\qquad |x|=c,\qquad
 |H'(h)|=1/\lambda.
\]
Since $|(H^2)'(h)|\leq1$, the chain rule gives
$|H'(x)|\leq\lambda$.

On $|z|=c$, \eqref{tk:eq:derivative} gives
\[
 Q(z)^2H'(z)=2Az+B+\varepsilon(z),
 \qquad |\varepsilon(z)|\leq\lambda.
\]
Evaluating at $x$, where $|Q(x)|=1$ and $|2Ax|=c/2$,
we obtain
\begin{equation}
\label{sp:eq:Bbound}
 |B|\leq\max\{\lambda,c/2\}.
\end{equation}

Next we prove that $\lambda\geq1/4$.
For distinct $z,w\in\DD(\zeta,\lambda)$, the term $D/z$
in $H$ contributes $-D/(zw)$ to the difference quotient
$(H(z)-H(w))/(z-w)$, with norm $c^2/\lambda$.
The difference quotient of $H(z)-D/z$ has norm at most
$\max\{\lambda,|B|,c^2\}$, and
\[
 \max\{\lambda,|B|,c^2\}<c<\frac{c^2}{\lambda}.
\]

Thus the pole term strictly dominates.
For every $|y|<c^2$, the equation $H(z)=y$ is equivalent to $(A-yF)z^3+(B-yE)z^2+(C-y)z+D=0,$
which has exactly one solution in $\DD(\zeta,\lambda)$.
Thus $H$ maps $\DD(\zeta,\lambda)$ bijectively
onto $\DD(0,c^2)$, and
\[
 |H(z)-H(w)|=\frac{c^2}{\lambda}|z-w|
 \qquad\text{for all }z,w\in\DD(\zeta,\lambda).
\]

Suppose that $\lambda<1/4$.
Since $c<\lambda^{1/2}<1/2$,
\eqref{sp:eq:Bbound} gives
$|B|<\lambda^{1/2}/2$.
On $|z|=\lambda^{1/2}$, the term $2Az$
in \eqref{tk:eq:derivative} has norm $\lambda^{1/2}/2$,
while the other terms have norm at most
\[
 \max\{\lambda,|B|,c^2\}<\lambda^{1/2}/2.
\]
Thus $|H'(z)|=\lambda^{1/2}/2$ on this circle.

Choose $x_0$ satisfying
\[
 H(x_0)=\zeta,\qquad |x_0|=\lambda^{1/2},
\]
and then choose $x_1$ satisfying
\[
 H(x_1)=x_0,\qquad |x_1|=c^2\lambda^{1/2}.
\]
The reduction at $|z|=\lambda^{1/2}$ and
Lemma~\ref{tk:lem:poledisk} give
\[
 \DD(x_1,c^2\lambda^{1/2})
 \xrightarrow{H}\DD(x_0,\lambda^{1/2})
 \xrightarrow{H}\DD(\zeta,\lambda)
 \xrightarrow{H}\DD(0,c^2).
\]
The middle map has degree two, and the other
maps are bijective.
Since every point in the first disk has absolute value
$c^2\lambda^{1/2}<c^2$, Lemma~\ref{tk:lem:count}
gives a fixed point $h$ of $H^3$ in that disk.
The derivative norms along its orbit are $c^{-2}$,
$\lambda^{1/2}/2$, and $c^2/\lambda$. Hence
\[
 |(H^3)'(h)|
 =c^{-2}\cdot\frac{\lambda^{1/2}}2
 \cdot\frac{c^2}{\lambda}
 =\frac1{2\lambda^{1/2}}>1,
\]
a contradiction. Thus $\lambda\geq1/4$.

Finally, $c<\lambda^{1/2}\leq2\lambda$, so $c/2<\lambda$.
Equation~\eqref{sp:eq:Bbound} gives $|B|\leq\lambda$.
\end{proof}

We now define
\[
 \phi(z)=\max\{|D|,\lambda\delta(z),\delta(z)^2\}.
\]
This function is bounded above and bounded away
from zero. It remains to prove
\[
 H^\#(z)\phi(z)\leq\phi(H(z)).
\]
Recall that in Case~1,
\[
 \lambda^3<|D|<\lambda^2,\qquad
 |C|=|D|/\lambda.
\]

For $z\in\mathcal E$, forward invariance and
Lemma~\ref{tk:lem:derivativeestimates} give
$\phi(z)=\phi(H(z))=1$ and $H^\#(z)\leq1$.
Thus the required inequality holds there.

\begin{itemize}
\item If $|z|\leq|D|$, the pole term gives
$H(z)\in\mathcal E$, $\phi(z)=|D|$, and
$H^\#(z)=|D|^{-1}$.
At zero, use $1/H(z)=zQ(z)/P(z)$.
Hence
\[
 H^\#(z)\phi(z)=1=\phi(H(z)).
\]

\item For $|D|<|z|<\lambda$,
Lemma~\ref{sp:lem:coefficientbounds} gives $|B|\leq\lambda$.
Thus the pole term strictly dominates both the value
and derivative:
\[
 |H(z)|=\frac{|D|}{|z|},\qquad
 |H'(z)|=\frac{|D|}{|z|^2}.
\]
Both $z$ and $H(z)$ lie in the open unit disk,
so $H^\#(z)=|H'(z)|$. Since
$\phi(z)=\max\{|D|,\lambda|z|\}$,
\[
 H^\#(z)\phi(z)
 =\max\{|H(z)|^2,\lambda|H(z)|\}
 =\phi(H(z)).
\]

\item For $\lambda\leq|z|\leq(|D|/\lambda)^{1/2}$,
Lemma~\ref{sp:lem:coefficientbounds} gives
$|B|\leq\lambda$ and $\lambda\geq1/4$.
The first bound implies
$|H(z)|\leq|D|/\lambda<1$, while the second gives
\[
 \frac{|z|}{2}\leq\frac{\lambda^{1/2}}2\leq\lambda.
\]
Together with $|C|=|D|/\lambda<\lambda$ and
\eqref{tk:eq:derivativebound}, these bounds imply
\[
 |H'(z)|\leq\max\{\lambda,|D|/|z|^2\}.
\]
Here $\phi(z)=|z|^2$ and $H^\#(z)=|H'(z)|$, so
\[
 H^\#(z)\phi(z)
 \leq\max\{\lambda|z|^2,|D|\}
 =|D|\leq\phi(H(z)).
\]

\item For $(|D|/\lambda)^{1/2}<|z|<1$,
Lemma~\ref{sp:lem:coefficientbounds} gives
$|B|\leq\lambda$ and $\lambda\geq1/4$.
The first bound ensures that the quadratic term
strictly dominates the value:
\[
 |H(z)|=|z|^2.
\]
The second gives
\[
 |z|/2\leq\max\{|z|^2,\lambda\}.
\]
Indeed, if $|z|\geq1/2$, then $|z|/2\leq|z|^2$;
otherwise $|z|/2<1/4\leq\lambda$.
Thus \eqref{tk:eq:derivativebound} gives
\[
 |H'(z)|\leq\max\{|z|^2,\lambda\}.
\]
Since $\phi(z)=|z|^2$ and $H^\#(z)=|H'(z)|$,
we obtain
\[
 H^\#(z)\phi(z)
 \leq\max\{|z|^4,\lambda|z|^2\}
 =\phi(H(z)).
\]
\end{itemize}

Since $|D|\leq\phi(z)\leq1$, the chain rule gives
\[
 (H^n)^\#(z)
 \leq\frac{\phi(H^n(z))}{\phi(z)}
 \leq|D|^{-1}.
\]
This proves the required bound in Case~1.

\subsection{Case 2: \texorpdfstring{$|D|\geq|C|^{3/2}$}{|D| >= |C|\string^(3/2)}}
\label{bd:sec:core}

Put $d=|D|^{1/3}$. Proposition~\ref{tk:prop:coefficientbounds}
gives
\[
 |B|\leq d,\qquad |C|\leq d^2,\qquad 0<d<1.
\]
For $0<|z|<d$, the term $D/z$ determines the absolute
value of $H(z)$. For $d<|z|<1$, the term $Az^2$ determines
it:
\[
 |H(z)|=
 \begin{cases}
  d^3/|z|,&0<|z|<d,\\
  |z|^2,&d<|z|<1.
 \end{cases}
\]
On $|z|=d$, cancellation is possible, but $|H(z)|\leq d^2$.

We will bound the derivative along each return to $|z|=d$.
When $d\geq1/4$, the bounds $|B|\leq d$ and $|C|\leq d^2$
suffice. When $d<1/4$, we also need to bound the distances
between the roots of $P$. These return estimates then give
a uniform bound on the spherical derivatives of all iterates.
\subsubsection{The additional restrictions when $d<1/4$}
\label{lc:sec:cube}

Let $R$ be the largest distance between two roots of $P$.

\begin{lemma}
\label{bd:lem:tripleroot}
\label{lc:lem:coefficientbounds}
All three roots of $P$ have absolute value $d$.
If $d<1/4$, then $R<d$ and
\begin{equation}
\label{lc:eq:cubecriterion}
 d\geq2^{-6},\qquad
 R^2\leq\min\{2d^{5/2},4d^{11/4}\}.
\end{equation}
\end{lemma}

\begin{proof}
The Newton polygon gives three roots of $P$ with absolute
value $d$, so $R\leq d$.
Assume that $d<1/4$.

If $R>0$, one root $\zeta$ is at distance $R$ from both other
roots. The factorization of $P$ then gives $|P'(\zeta)|=R^2$.
For every $z\in\DD(\zeta,R)$, the distances from $z$ to
the other two roots of $P$ are both $R$, and $|zQ(z)|=d$.
Thus the linear term in the expansion of $H$ at $\zeta$
dominates, and
\[
 H:\DD(\zeta,R)\longrightarrow\DD(0,R^3/d)
 \quad\text{is bijective, with }|H'|=R^2/d.
\]

When $R>0$, we use this disk to construct periodic points.
For $d<|x|<1$, the quadratic reduction gives a surjective map
\[
 H:\DD(x,|x|)\longrightarrow\DD(H(x),|x|^2)
\]
of mapping degree two.
Indeed, for every $y\in\DD(H(x),|x|^2)$, the equation
$H(z)=y$ has two roots with absolute value $|x|$,
while its third root has smaller absolute value.
Writing $z=xu$, the reduction of $H(xu)/H(x)$ is $\bar u^2$.
Since $y/H(x)$ reduces to one, the two roots satisfy
$\bar u^2=1$.
In characteristic two this implies $\bar u=1$, so both
roots lie in $\DD(x,|x|)$.

For $k=1$ or $2$ with $d^{2^{-k}}<1/2$, pull back
$\DD(\zeta,R)$ through the disks whose points have absolute values
$d^{1/2},\ldots,d^{2^{-k}}$.
At each step, Lemma~\ref{tk:lem:refine} gives disks in the
inverse image that map surjectively onto the chosen disk.
Choose one such disk and continue taking inverse images.
Finally, Lemma~\ref{tk:lem:poledisk} gives a disk near zero
that maps bijectively onto the last chosen disk.
All points in the
first source disk have absolute value $d^{3-2^{-k}}$.
The derivative along the step near the pole has norm
$d^{-3+2^{1-k}}$.
At the steps whose input points have absolute value
$d^{2^{-j}}$, the derivative norms are $d^{2^{-j}}/2$.

Consequently, if
\[
 d^{3-2^{-k}}<R^3/d,
\]
Lemma~\ref{tk:lem:count} gives a fixed point of $H^{k+2}$
whose multiplier has norm
\begin{equation}
\label{lc:eq:returntest}
 \frac{R^2}{2^k d^{3-2^{-k}}}.
\end{equation}
If $R=d$, taking $k=1$ gives $1/(2d^{1/2})>1$.
Hence $R<d$.

Now suppose that $d<2^{-6}$.
If $R<d^{7/6}$, all three roots lie in $\DD(\zeta,d)$,
where $\zeta$ is any root, and
\[
 H:\DD(\zeta,d)\longrightarrow\DD(0,d^2)
\]
has mapping degree three.
Choose a disk whose points have absolute value $d^{1/2}$
and which maps onto $\DD(\zeta,d)$.
Then choose a disk near zero that maps bijectively onto that
disk. Lemma~\ref{tk:lem:count} gives a fixed point
$h$ of $H^3$ with $|h|=d^{5/2}$.

At $z=H^2(h)$, we have
\[
 |P(z)|=d|h|=d^{7/2}>R^3.
\]
Thus the distances from $z$ to all three roots of $P$
equal $d^{7/6}$.
Since $|3|=1$, this gives $|P'(z)|=d^{7/3}$ and
$|H'(z)|=d^{4/3}$.
In the identity
\[
 H'(z)=\frac{P'(z)}{zQ(z)}
       -H(z)\left(\frac1z+\frac{Q'(z)}{Q(z)}\right),
\]
the second term has norm at most
$|h|/d=d^{3/2}<d^{4/3}$.
The multiplier of $H^3$ at $h$ has norm
\[
 d^{-2}\cdot\frac{d^{1/2}}2\cdot d^{4/3}
 =\frac1{2d^{1/6}}>1.
\]

If $R\geq d^{7/6}$, take $k=2$ in \eqref{lc:eq:returntest}.
Here $R^3/d\geq d^{5/2}>d^{11/4}$, and the multiplier is at least
\[
 \frac{d^{7/3}}{4d^{11/4}}
 =\frac1{4d^{5/12}}>1.
\]
Both cases are impossible, so $d\geq2^{-6}$.

Finally, choose $k=1$ if $d\geq1/16$, and $k=2$ if $d<1/16$.
If $R^2>2^k d^{3-2^{-k}}$, then
\[
 \frac{R^3/d}{d^{3-2^{-k}}}
 >2^{3k/2}d^{(1-2^{-k})/2}\geq1,
\]
where the last inequality uses $d\geq2^{-6}$.
Equation~\eqref{lc:eq:returntest} would give a repelling point.
Thus
\[
 R^2\leq
 \begin{cases}
  2d^{5/2},&1/16\leq d<1/4,\\
  4d^{11/4},&2^{-6}\leq d<1/16.
 \end{cases}
\]
These inequalities give \eqref{lc:eq:cubecriterion}.
\end{proof}

\begin{lemma}
\label{lc:lem:terminalderivative}
For $|z|=d$, we have
\begin{equation}
\label{lc:eq:terminalbound}
 |H'(z)|\leq
 \max\{d^{-1/3}|H(z)|^{2/3},R^2/d\}.
\end{equation}
\end{lemma}

\begin{proof}
Fix a root $\zeta$ of $P$. If $|z-\zeta|>R$, the distances
from $z$ to all three roots of $P$ equal $|z-\zeta|$.
Hence $|P(z)|=|z-\zeta|^3$ and
$|P'(z)|=|z-\zeta|^2$.
If $|z-\zeta|\leq R$, then
\[
 |P(z)|\leq R^3,\qquad |P'(z)|\leq R^2.
\]

Using
\[
 H'(z)=\frac{P'(z)}{zQ(z)}
       -H(z)\left(\frac1z+\frac{Q'(z)}{Q(z)}\right),
\]
we bound the first term by the maximum in
\eqref{lc:eq:terminalbound}, since $|zQ(z)|=d$.
The second term has norm at most $|H(z)|/d$, and
\[
 |H(z)|/d\leq d^{-1/3}|H(z)|^{2/3}
\]
because $|H(z)|\leq d^2$.
\end{proof}

\subsubsection{A uniform bound for the iterate derivatives}

We now treat both ranges of $d$. We estimate the derivative
product between successive visits to $|z|=d$.
We then control the portion before the first visit to this
circle and the portion after the last visit.

\begin{lemma}
\label{lc:lem:uniformderivative}
For every $z\in\PP(\Ctwo)$ and $n\geq0$,
\[
 (H^n)^\#(z)\leq d^{-4}.
\]
\end{lemma}

\begin{proof}
Put $\phi_0(z)=\max\{d^3,\delta(z)\}$.
Away from $|z|=d$, we have
\begin{equation}
\label{lc:eq:offcircle}
 H^\#(z)\phi_0(z)\leq\phi_0(H(z)).
\end{equation}
For $|z|<d$, this follows from the formulas near the pole,
including zero in the coordinate $1/H$.
For $d<|z|<1$, it follows from $|H(z)|=|z|^2$ and
\[
 |H'(z)|\leq\max\{d,|z|/2,|z|^2\}\leq|z|.
\]
On $\mathcal E$, use
Lemma~\ref{tk:lem:derivativeestimates}.
If $|H^j(w)|\ne d$ for $0\leq j<m$, multiplying
\eqref{lc:eq:offcircle} gives $(H^m)^\#(w)\leq d^{-3}$.
On the circle itself,
\begin{equation}
\label{lc:eq:circleloss}
 |H'(z)|\leq d.
\end{equation}
Since $\phi_0(z)=d$ and $\phi_0(H(z))\geq d^3$ on this
circle, we also have
\[
 H^\#(z)\frac{\phi_0(z)}{\phi_0(H(z))}\leq d^{-1}.
\]

Let $|z|=d$, and suppose that $N$ is the first return time:
\[
 |H^N(z)|=d,\qquad |H^j(z)|\ne d
 \quad(1\leq j<N).
\]
We prove that $(H^N)^\#(z)\leq1$.
Since $|z|=|H^N(z)|=d<1$,
\[
 (H^N)^\#(z)=|(H^N)'(z)|.
\]
For $0<|z|<d$, the image has absolute value $d^3/|z|$;
for $d<|z|<1$, it has absolute value $|z|^2$.
These formulas imply that $N=k+2$ for some $k\geq0$,
with successive absolute values
\[
 d\longmapsto d^{3-2^{-k}}\longmapsto d^{2^{-k}}
 \longmapsto\cdots\longmapsto d,\qquad k\geq0.
\]
An orbit entering $\mathcal E$ or the invariant
annulus $d^2<|z|<d$ cannot return.
There are $k$ steps whose input points satisfy $d<|z|<1$.
When $k=0$, the step from absolute value $d^2$ to absolute
value $d$ already returns to the circle.

At a step beginning at a point $w$ with $|w|=d^{2^{-j}}$,
\eqref{tk:eq:derivativebound} gives
\[
 |H'(w)|\frac{|w|}{|H(w)|}
 \leq\max\{1/2,d^{2^{-j}}\}.
\]
At a step beginning at a point $w$ with $0<|w|<d$,
we have $|H'(w)|=|H(w)|/|w|$.

If $d\geq1/4$, all the absolute values $d^{2^{-j}}$,
for $1\leq j\leq k$, are at least $1/2$.
Using \eqref{lc:eq:circleloss}, we obtain
\[
 (H^{k+2})^\#(z)
 \leq d^{-1+2^{-k}}\prod_{j=1}^{k}d^{2^{-j}}
 =1.
\]
An empty product equals one.

If $d<1/4$, the derivative estimate \eqref{lc:eq:terminalbound}
on $|z|=d$ gives
\begin{equation}
\label{lc:eq:firstreturnbound}
 (H^{k+2})^\#(z)\leq
 \left(\prod_{j=1}^{k}\max\{1/2,d^{2^{-j}}\}\right)
 \max\{d^{-(1-2^{-k})/3},R^2d^{-3+2^{-k}}\}.
\end{equation}
For $k=0$, this is at most $\max\{1,R^2/d^2\}=1$.
For $k=1$, it is at most
\[
 \max\left\{\frac1{2d^{1/6}},
             \frac{R^2}{2d^{5/2}}\right\}\leq1
\]
by \eqref{lc:eq:cubecriterion}.
For $k\geq2$, the lower bound $d\geq2^{-6}$ gives
\[
 \prod_{j=1}^{k}\max\{1/2,d^{2^{-j}}\}
 =\max\{1/4,d^{1/4}/2\}\,d^{1/4-2^{-k}}.
\]
Substitution into \eqref{lc:eq:firstreturnbound} gives
\[
 (H^{k+2})^\#(z)
 \leq\max\left\{
 \frac1{4d^{1/4}},\frac12,
 \frac{R^2}{4d^{11/4}},\frac{R^2}{2d^{5/2}}
 \right\}\leq1.
\]
Thus the derivative product along every complete return is at most one.

We now split a finite orbit
$z,H(z),\ldots,H^n(z)$ at its visits to $|z|=d$.
Before the first visit, every input point is off the circle.
More generally, if $|H^j(w)|\ne d$ for $0\leq j<m$,
multiplying \eqref{lc:eq:offcircle} gives
\[
 (H^m)^\#(w)\frac{\phi_0(w)}{\phi_0(H^m(w))}\leq1.
\]
If $N$ is the first return time of $w$, then
$\phi_0(w)=\phi_0(H^N(w))=d$ and $(H^N)^\#(w)\leq1$.
Therefore
\[
 (H^N)^\#(w)\frac{\phi_0(w)}{\phi_0(H^N(w))}\leq1.
\]

After the last visit, the starting point is the only input
point on the circle. At that step,
$H^\#(w)\phi_0(w)/\phi_0(H(w))\leq d^{-1}$.
At every later step, \eqref{lc:eq:offcircle} bounds this
expression by one.
An empty portion contributes the product one.
If the orbit never visits the circle, \eqref{lc:eq:offcircle}
applies at every step.
Thus the chain rule gives
\[
 (H^n)^\#(z)\frac{\phi_0(z)}{\phi_0(H^n(z))}\leq d^{-1}.
\]
Since $d^3\leq\phi_0\leq1$, we conclude that
\[
 (H^n)^\#(z)\leq d^{-4}
 \qquad\text{for all }z\in\PP(\Ctwo)\text{ and }n\geq0.
\]
\end{proof}

This proves the required bound in Case~2.

\subsection{Case 3: \texorpdfstring{$|D|=|C|^2$}{|D| = |C|\string^2}}
\label{sec:equality-start}

Put $c=|C|^{1/2}$. Then
\[
 |C|=c^2,\qquad |D|=c^4,\qquad
 0<c<1,\qquad |B|<c.
\]
Recall that the numerator of $H$ is
\[
 P(z)=Az^3+Bz^2+Cz+D.
\]
The Newton polygon of $P$ shows that $P$ has one zero
$\gamma$ with absolute value $c^2$ and two zeros
$\zeta_1,\zeta_2$ with absolute value $c$.
Throughout this case, $\gamma,\zeta_1,\zeta_2$ denote these
zeros of $P$. They are also zeros of $H$, since the denominator
$zQ(z)$ does not vanish at any of them.

Away from the circles $|z|=c^2$ and $|z|=c$,
\[
 |H(z)|=
 \begin{cases}
  c^4/|z|,&0<|z|<c^2,\\
  c^2,&c^2<|z|<c,\\
  |z|^2,&c<|z|<1.
 \end{cases}
\]
We first obtain a further restriction on the coefficient $B$.
We then show that the first restriction is sufficient to apply
the construction from Case~1 and obtain the uniform derivative
bound for $c\geq1/2$.

For $c<1/2$, we need additional restrictions on $c$ and on
the locations of the roots. In this range, we will bound the
distance $|\zeta_1-\zeta_2|$, identify the disk
$\DD(\gamma,c^2)$ with $\DD(C,c^2)$, and bound $|C-\gamma|$.
These results will allow us to construct a function $\phi$
bounded above and bounded away from zero, and prove the
uniform derivative bound.

\subsubsection{The restriction on $B$}

\begin{lemma}
\label{bd:lem:equalityB}
\label{bd:lem:plateaugeometry}
\label{mx:lem:disks}
We have
\[
 \begin{cases}
  |B|\leq c^2,&c\geq1/2,\\
  |B|=c/2,&c<1/2.
 \end{cases}
\]
\end{lemma}

\begin{proof}
For every $|y|<c^2$, clearing denominators in the equation
$H(z)=y$ gives
\begin{equation}
\label{bd:eq:fiber}
 (A-yF)z^3+(B-yE)z^2+(C-y)z+D=0.
\end{equation}
The polynomial on the left has one root with absolute value
$c^2$ and two roots with absolute value $c$.
To locate the latter two roots, write $z=\zeta_1u$.
Since $H(\zeta_1)=0$, the coefficient $A\zeta_1^2/C$
reduces to one. On $|u|=1$, the reduction of
$H(\zeta_1u)/C$ is therefore $\bar u^2+1$.
For $|y|<c^2$, reducing the equation $H(\zeta_1u)=y$ gives
\[
 \bar u^2+1=0.
\]
In characteristic two this implies $\bar u=1$, or
$|z-\zeta_1|<c$.
Thus, for every $|y|<c^2$, the equation $H(z)=y$ has
two roots in $\DD(\zeta_1,c)$.
The same reduction shows that $H$ sends this disk into
$\DD(0,c^2)$.
Consequently,
\[
 H:\DD(\zeta_1,c)\longrightarrow\DD(0,c^2)
\]
is surjective of mapping degree two.

For the equation $H(z)=\zeta_1$, we have $|C-\zeta_1|=c$.
The Newton polygon segment joining the constant and linear
coefficients of the polynomial on the left of
\eqref{bd:eq:fiber}, with $y=\zeta_1$, therefore gives a root
$\xi$ with $|\xi|=c^3$.
Since $c^4>c^6$, Lemma~\ref{tk:lem:poledisk} gives a bijection
\[
 H:\DD(\xi,c^3)\longrightarrow\DD(\zeta_1,c)
\]
with constant derivative norm $c^{-2}$.
Consequently,
\[
 \DD(\xi,c^3)
 \xrightarrow{H}\DD(\zeta_1,c)
 \xrightarrow{H}\DD(0,c^2).
\]
Since $c^3<c^2$, Lemma~\ref{tk:lem:count} gives a fixed point
$h$ of $H^2$ in the first disk. Put $x=H(h)$. Then
\[
 |h|=c^3,\qquad |x|=c,\qquad |H'(h)|=c^{-2}.
\]
The absence of repelling periodic points implies
$|H'(x)|\leq c^2$.

On $|z|=c$, \eqref{tk:eq:derivative} gives
\[
 Q(z)^2H'(z)=2Az+B+\varepsilon(z),
 \qquad |\varepsilon(z)|\leq c^2.
\]
Here $|Q(x)|=1$ and $|2Ax|=c/2$.
If $c\geq1/2$, then $c/2\leq c^2$, so $|B|>c^2$
would force $|H'(x)|=|B|>c^2$, a contradiction.
If $c<1/2$, then $c/2>c^2$. Unless $|B|=c/2$,
one of $2Ax$ and $B$ strictly dominates, giving
$|H'(x)|=\max\{c/2,|B|\}>c^2$, again a contradiction.
\end{proof}

Suppose now that $c\geq1/2$. The restriction $|B|\leq c^2$
is sufficient to obtain the uniform derivative bound using
the construction from Case~1. Define
\[
 \phi(z)=\max\{c^4,c^2\delta(z),\delta(z)^2\}.
\]
The verification in Section~\ref{sp:sec:plateau}
applies with $\lambda=c^2$:
\[
 \lambda\geq1/4,\qquad |B|\leq\lambda,\qquad
 \lambda^3<|D|=\lambda^2,\qquad |C|=|D|/\lambda.
\]
The four regional estimates remain valid when $|D|=\lambda^2$
and give $H^\#\phi\leq\phi\circ H$.
Since $c^4\leq\phi\leq1$, the chain rule yields
\[
 (H^n)^\#(z)
 \leq\frac{\phi(H^n(z))}{\phi(z)}
 \leq c^{-4}=|D|^{-1}.
\]
It remains to prove this bound when $c<1/2$.

\subsubsection{Additional restrictions when $c<1/2$}

Assume henceforth that $c<1/2$, so $|B|=c/2$.
To control the derivative near $\zeta_1$ and $\zeta_2$, put
\[
 R=|\zeta_1-\zeta_2|.
\]
We will prove that $c\geq2^{-3/2}$ and $R\leq c^2$,
and that $\DD(C,c^2)=\DD(\gamma,c^2)$ with
$|C-\gamma|\leq4c^4$.
We first bound $|\zeta_1-\zeta_2|$. We then
identify the disk containing $C$ and obtain the quantitative
bound on its distance from $\gamma$.
On $|z|=c$, write
\begin{equation}
\label{mx:eq:factor}
 H(z)=\frac{A(z-\gamma)}{zQ(z)}g(z),
 \qquad g(z)=(z-\zeta_1)(z-\zeta_2).
\end{equation}
On this circle, the factor $A(z-\gamma)/(zQ(z))$
has norm one, and its derivative has norm at most one.
Differentiating the product gives
\[
 H'(z)=\frac{A(z-\gamma)}{zQ(z)}g'(z)
       +\left(\frac{A(z-\gamma)}{zQ(z)}\right)'g(z).
\]

\begin{lemma}
\label{bd:thm:smallpole}
We have $c\geq1/4$ and $R\leq c^2$.
\end{lemma}

\begin{proof}
Let $h$ be a fixed point of $H^2$ constructed in the proof of
Lemma~\ref{bd:lem:equalityB}, and put $x=H(h)$. Then
\[
 |g(x)|=c^3,\qquad |g'(x)|\leq c^2.
\]
The second inequality follows from
$|H'(x)|\leq c^2$ and \eqref{mx:eq:factor}.

If $|x-\zeta_1|\leq R$, then
\[
 |g'(x)|
 =|2(x-\zeta_1)+(\zeta_1-\zeta_2)|
 =R.
\]
This would give $R\leq c^2$ and
$|g(x)|\leq c^4$, a contradiction.
Thus
\[
 |x-\zeta_1|=|x-\zeta_2|=c^{3/2}.
\]

Suppose that $c<1/4$.
Since $c^{3/2}/2>c^2$, the derivative bound
forces $R=c^{3/2}/2$.
For $m=(\zeta_1+\zeta_2)/2$,
\[
 g(z)=(z-m)^2-\frac{(\zeta_1-\zeta_2)^2}{4},
 \qquad |g(m)|=c^3.
\]
Moreover, $|m-\zeta_1|=2R=c^{3/2}<c$,
so $\DD(m,c)=\DD(\zeta_1,c)$.

For $c<s<1$, let $a$ satisfy $|a|=s^2$.
For every $y\in\DD(a,s^2)$, the equation $H(z)=y$
has two roots with absolute value $s$.
Choose a solution $b$ of $H(b)=a$ with $|b|=s$
and write $z=bu$.
On $|u|=1$, the reduction of $H(bu)/a$ is $\bar u^2$,
while $y/a$ reduces to one.
The two roots therefore satisfy $\bar u^2=1$.
In characteristic two this implies $\bar u=1$, so both
roots lie in $\DD(b,s)$.
Thus $H$ maps $\DD(b,s)$ onto $\DD(a,s^2)$
with mapping degree two.
Apply this with $s=c^{1/2}$ and $a=\zeta_1$.
Then choose a disk near zero that $H$ maps bijectively
onto the resulting disk $\DD(x_0,c^{1/2})$. This gives
\[
 \DD(x_1,c^{7/2})
 \xrightarrow{H}\DD(x_0,c^{1/2})
 \xrightarrow{H}\DD(\zeta_1,c)
 \xrightarrow{H}\DD(0,c^2),
 \qquad |x_1|=c^{7/2}<c^2.
\]
Lemma~\ref{tk:lem:count} gives a fixed point
$h$ of $H^3$ in the first disk.
For $v=H^2(h)$, we have $|g(v)|=c^{7/2}$ and
\[
 g(v)=(v-m)^2-\frac{(\zeta_1-\zeta_2)^2}{4},
 \qquad
 |g(v)|<c^3=
 \left|\frac{(\zeta_1-\zeta_2)^2}{4}\right|.
\]
Hence
\[
 |v-m|=c^{3/2},\qquad |g'(v)|=c^{3/2}/2.
\]
The derivative of $A(z-\gamma)/(zQ(z))$ at $v$,
multiplied by $g(v)$, has norm at most $c^{7/2}$.
Hence $|H'(v)|=c^{3/2}/2$.
The multiplier of $H^3$ at $h$ has norm
\[
 c^{-3}\cdot\frac{c^{1/2}}2\cdot\frac{c^{3/2}}2
 =\frac1{4c}>1,
\]
a contradiction. Therefore $c\geq1/4$.

Now $|2(x-\zeta_1)|=c^{3/2}/2\leq c^2$.
Together with $|g'(x)|\leq c^2$, this gives
$R\leq c^2$.
\end{proof}

We next determine the location of the disk $\DD(\gamma,c^2)$.
For $c^2<|z|<c$, we have $|H(z)-C|<c^2$.
We must therefore show that $C$ lies in
$\DD(\gamma,c^2)$ and control its distance
from $\gamma$.

\begin{lemma}
\label{bd:lem:centralrestrictions}
We have
\[
 \overline{D/C^2}=1,\qquad
 \DD(C,c^2)=\DD(\gamma,c^2),\qquad
 |H(C)|=|C-\gamma|<c^2.
\]
\end{lemma}

\begin{proof}
For distinct $z,w\in\DD(a,c^2)$ with $|a|=c^2$,
the term $D/z$ gives $-D/(zw)$ in the difference quotient
$(H(z)-H(w))/(z-w)$.
This term has norm one.
The difference quotient of $H(z)-D/z$ has norm less than one.
The power-series expansion therefore shows that $H$ maps this
disk bijectively onto $\DD(H(a),c^2)$ and preserves distances.
In particular, $\DD(\gamma,c^2)$ maps onto $\DD(0,c^2)$.

In the coordinate $w=z/C$, the reduced map
on $|z|=c^2$ is
\begin{equation}
\label{mx:eq:M}
 w\longmapsto1+\frac{\overline{D/C^2}}{w}.
\end{equation}
If $\overline{D/C^2}\ne1$, the residue of
$\gamma/C$ is different from $1$.
Write $z=\zeta_1u$ on $|z|=c$.
As in the proof of Lemma~\ref{bd:lem:equalityB}, the
reduction of $H(\zeta_1u)/C$ is $\bar u^2+1$.
The equation
\[
 \bar u^2+1=\overline{\gamma/C}
\]
has a unique nonzero root of multiplicity two, since
$\overline{\gamma/C}\ne1$.
This root is different from one because
$\overline{\gamma/C}\ne0$.
It therefore determines a disk $\DD(x_0,c)$ disjoint from
$\DD(\zeta_1,c)$, on which $H$ maps onto
$\DD(\gamma,c^2)$ with mapping degree two.
Choosing a disk near zero that $H$ maps bijectively onto
$\DD(x_0,c)$ gives
\[
 \DD(x_1,c^3)
 \xrightarrow{H}\DD(x_0,c)
 \xrightarrow{H}\DD(\gamma,c^2)
 \xrightarrow{H}\DD(0,c^2),
 \qquad |x_1|=c^3<c^2.
\]
For every $z\in\DD(x_0,c)$, we have
$|z-\zeta_1|=|z-\zeta_2|=c$. Since
$R\leq c^2<c/2$, we have $|g'|=c/2$.
The derivative of $A(z-\gamma)/(zQ(z))$, multiplied
by $g(z)$, has norm at most $c^2$.
Thus $|H'|=c/2$ throughout this disk.
Lemma~\ref{tk:lem:count} gives a fixed point $h$ of $H^3$
in $\DD(x_1,c^3)$, and the multiplier of $H^3$ at $h$
has norm
\[
 c^{-2}\cdot(c/2)\cdot1=\frac1{2c}>1,
\]
a contradiction. Hence $\overline{D/C^2}=1$.

Applying \eqref{mx:eq:M} to $\gamma/C$ and using
$H(\gamma)=0$ and $\overline{D/C^2}=1$, we obtain
\[
 0=1+\frac1{\overline{\gamma/C}}.
\]
In characteristic two, this gives $\overline{\gamma/C}=1$,
so
\[
 |C-\gamma|<c^2,\qquad
 \DD(C,c^2)=\DD(\gamma,c^2).
\]
Distance preservation on this disk yields
\[
 |H(C)|=|H(C)-H(\gamma)|=|C-\gamma|<c^2.
\]
\end{proof}

The next composition of disk maps gives the quantitative bounds
needed for the definition of $\phi$.

\begin{proposition}
\label{bd:prop:collision}
\label{mx:prop:four}
We have
\[
 c\geq2^{-3/2},\qquad |C-\gamma|\leq4c^4.
\]
\end{proposition}

\begin{proof}
Put
\[
 r_0=c^{4/3},\qquad
 L=\max\{|C-\gamma|,r_0^2\},\qquad r=c^4/L.
\]
Then $c^2<r\leq r_0$.
We first find a bijective map
\[
 H:\DD(v,r)\longrightarrow\DD(\gamma,L)
\]
with $|v|=r$ and derivative norm $L/r$.

If $|C-\gamma|>r_0^2$, then $L=|C-\gamma|$
and $r<r_0$. For every $y\in\DD(\gamma,L)$,
\[
 |A-yF|=1,\qquad |B-yE|=c/2,\qquad |C-y|=L.
\]
Together with
\[
 r^3<c^4,\qquad (c/2)r^2<c^4,
\]
these bounds give a unique root of \eqref{bd:eq:fiber}
with absolute value $r$. Choose the solution $v$ of
$H(v)=\gamma$ with $|v|=r$.
The linear term in the expansion of $H$ at $v$ dominates
on $\DD(v,r)$, giving the bijective map onto $\DD(\gamma,L)$
with constant derivative norm $c^4/r^2=L/r$.

If $|C-\gamma|\leq r_0^2$, then $L=r_0^2$
and $r=r_0$. Choose $\beta$ with
$(A-\gamma F)\beta^3=D$.
After substituting $z=\beta u$ into the polynomial equation
obtained from $H(z)=\gamma$ and dividing by $D$,
the reduced equation is
\[
 u^3+tu+1=0,\qquad
 t=\overline{(C-\gamma)\beta/D}.
\]
The quadratic coefficient has norm
$1/(2c^{1/3})<1$, because $c\geq1/4$.
In characteristic two,
\[
 (u^3+tu+1)-u(3u^2+t)=1,
\]
so every root is simple.
Hensel's lemma lifts a root of the reduced equation
to a solution $u$ of the equation before reduction.
Set $v=\beta u$. Then
\[
 H(v)=\gamma,\qquad |v|=r.
\]
For $|y-\gamma|<L$, the polynomials obtained by clearing
denominators in $H(z)=y$ and $H(z)=\gamma$ differ by
$(y-\gamma)zQ(z)$. This difference has norm less than
$Lr=r^3$ on $|z|=r$.
For every $y\in\DD(\gamma,L)$, the same simple root
of the reduced equation gives exactly one solution of
$H(z)=y$ in $\DD(v,r)$.
Hence $H$ maps $\DD(v,r)$ bijectively onto
$\DD(\gamma,L)$, with constant derivative norm $r=L/r$.

Suppose that $r<1/4$.
The construction in the proof of Lemma~\ref{bd:thm:smallpole},
with $s=r^{1/2}$ and $a=v$, gives a disk
$\DD(x_0,r^{1/2})$ mapping onto $\DD(v,r)$
with mapping degree two.
Choose a disk near zero that $H$ maps bijectively onto
$\DD(x_0,r^{1/2})$.
Since $H$ preserves distances on $\DD(\gamma,c^2)$,
it maps $\DD(\gamma,L)$ onto $\DD(0,L)$.
Together these maps give
\[
 \DD(x_1,c^4/r^{1/2})
 \xrightarrow{H}\DD(x_0,r^{1/2})
 \xrightarrow{H}\DD(v,r)
 \xrightarrow{H}\DD(\gamma,L)
 \xrightarrow{H}\DD(0,L).
\]
Here $c<r^{1/2}<1/2$ and
\[
 |x_1|=c^4/r^{1/2}=Lr^{1/2}<L.
\]
Thus Lemma~\ref{tk:lem:count} gives a fixed point $h$ of
$H^4$ in the first disk.
The multiplier of $H^4$ at $h$ has norm
\[
 |(H^4)'(h)|=\frac{r}{c^4}\cdot\frac{r^{1/2}}2
 \cdot\frac Lr\cdot1
 =\frac1{2r^{1/2}}>1,
\]
a contradiction. Hence $r\geq1/4$.
Since $r\leq c^{4/3}$ and
$|C-\gamma|\leq L=c^4/r$, the stated bounds follow.
\end{proof}

\subsubsection{Construction of $\phi$ and the uniform derivative bound}

Combining Lemmas~\ref{bd:lem:equalityB},
\ref{bd:thm:smallpole}, and~\ref{bd:lem:centralrestrictions}
with Proposition~\ref{bd:prop:collision}, we have
\[
 2^{-3/2}\leq c<1/2,\qquad
 |B|=c/2,\qquad
 R\leq c^2,\qquad |C-\gamma|\leq4c^4.
\]
We also have $\overline{D/C^2}=1$ and
$\DD(C,c^2)=\DD(\gamma,c^2)$.
These restrictions have two roles in the derivative estimates.
The bound on $|\zeta_1-\zeta_2|$ controls the derivative
near $\zeta_1$ and $\zeta_2$.
The equality of the two disks identifies the image of
the annulus $c^2<|z|<c$, and the bound on $|C-\gamma|$
controls the displacement of that image from $\gamma$.

We are now ready to construct $\phi$.
Since $H$ preserves distances on $\DD(\gamma,c^2)$, we use
the distance from $\gamma$ on this disk. Elsewhere, we use
the absolute value of $z$, truncated at one.
The lower bound on $c$ and the bound on $|C-\gamma|$ allow
us to use the same positive lower bound $4c^4$ in both expressions.
Put $\kappa=4c^4$ and define
\[
 \phi(z)=
 \begin{cases}
  \max\{\kappa,|z-\gamma|\},
       &z\in\DD(\gamma,c^2),\\
  \max\{\kappa,\delta(z),4\delta(z)^2\},
       &\text{otherwise}.
 \end{cases}
\]
On $\DD(\gamma,c^2)$, the identity $|H(z)|=|z-\gamma|$
relates the two values of $\phi$ directly.
The constant $\kappa$ ensures that cancellation among terms
whose norms are at most $\kappa$ does not affect the comparison.
In particular, $\kappa\leq\phi(z)\leq4$.

We verify $H^\#(z)\phi(z)\leq\phi(H(z))$.

\begin{itemize}
\item
On $\mathcal E$, forward invariance gives
$\phi(z)=\phi(H(z))=4$, and $H^\#(z)\leq1$.

If $|z|\leq c^4$, the pole term gives
$H(z)\in\mathcal E$, $\phi(z)=\kappa$, and
$H^\#(z)=c^{-4}$. Hence equality holds.
At zero, use the coordinate $1/H$.

\item
If $c^4<|z|<c^2$, then
$|H(z)|=c^4/|z|>c^2$ and
\[
 \begin{aligned}
 H^\#(z)\phi(z)
 &=\frac{c^4}{|z|^2}\max\{\kappa,|z|\}\\
 &=\max\{4|H(z)|^2,|H(z)|\}
 =\phi(H(z)).
 \end{aligned}
\]

\item
On $\DD(\gamma,c^2)$, distance preservation gives
\[
 |H'(z)|=1,\qquad |H(z)|=|z-\gamma|<c^2<1/4.
\]
Therefore
\[
 H^\#(z)\phi(z)
 =\max\{\kappa,|z-\gamma|\}
 =\phi(H(z)).
\]

If $|z|=c^2$ and $z\notin\DD(\gamma,c^2)$, then
$\overline{z/C}$ is nonzero and different from one.
The reduced map gives
\[
 \overline{H(z)/C}
 =1+\frac1{\overline{z/C}}\notin\{0,1\}.
\]
Thus $|H(z)|=c^2$ and $H(z)\notin\DD(\gamma,c^2)$.
Both $\phi(z)$ and $\phi(H(z))$ equal $c^2$, and $|H'(z)|=1$.

\item
If $c^2<|z|<c$, then $H(z)$ lies in
$\DD(C,c^2)=\DD(\gamma,c^2)$.
We have
\[
 H(z)-\gamma=
 \frac{(A-\gamma F)z^2+(B-\gamma E)z
       +(C-\gamma)+D/z}{Q(z)}.
\]
The four numerator terms have norms
\[
 |z|^2,\qquad (c/2)|z|,\qquad
 |C-\gamma|,\qquad c^4/|z|.
\]
The terms $(B-\gamma E)z$ and $C-\gamma$ have norms
at most $\kappa$.
The terms $(A-\gamma F)z^2$ and $D/z$ can cancel only when
$|z|=c^{4/3}$, where their common norm is
$c^{8/3}\leq\kappa$.
Consequently,
\[
 \phi(H(z))=\max\{\kappa,|z|^2,c^4/|z|\}.
\]
Equation~\eqref{tk:eq:derivativebound} gives
\[
 |H'(z)|\leq\max\{c/2,c^4/|z|^2\}.
\]
Since $\phi(z)=\max\{|z|,4|z|^2\}$, we obtain
\[
 \begin{aligned}
 H^\#(z)\phi(z)
 &\leq
 \max\{c|z|/2,2c|z|^2,c^4/|z|,\kappa\}\\
 &\leq\max\{\kappa,|z|^2,c^4/|z|\}
 =\phi(H(z)).
 \end{aligned}
\]
Here $c|z|/2\leq c^2/2\leq\kappa$ and $2c<1$.

\item
If $|z|=c$, put $\ell=|z-\zeta_1|\leq c$.
The factorization \eqref{mx:eq:factor} and
$R\leq c^2$ give
\[
 |H'(z)|\leq\max\{c^2,\ell/2\},\qquad
 \max\{\kappa,|H(z)|\}=\max\{\kappa,\ell^2\}.
\]
Indeed, if $\ell>R$, then
$|z-\zeta_1|=|z-\zeta_2|=\ell$. If $\ell\leq R$,
both $|g(z)|$ and $\ell^2$ are at most
$R^2\leq c^4\leq\kappa$.

The reduction of $H(z)/C$ is $1+\overline{Az^2/C}$.
Since $|Az^2/C|=1$, this reduction is different from one,
so $H(z)$ does not
lie in $\DD(\gamma,c^2)$. Thus
\[
 \phi(z)=4c^2,\qquad
 \phi(H(z))=\max\{\kappa,\ell^2\}.
\]
It follows that
\[
 H^\#(z)\phi(z)
 \leq\max\{\kappa,2c^2\ell\}
 \leq\max\{\kappa,\ell^2\}
 =\phi(H(z)).
\]
For the second inequality, use $\kappa$
when $\ell\leq2c^2$ and $\ell^2$ otherwise.

\item
If $c<|z|<1$, then
\[
 |H(z)|=|z|^2,\qquad
 |H'(z)|\leq|z|\max\{1/2,|z|\},\qquad
 \phi(z)=4|z|^2.
\]
Therefore
\[
 H^\#(z)\phi(z)
 \leq\max\{2|z|^3,4|z|^4\}
 \leq\max\{|z|^2,4|z|^4\}
 =\phi(H(z)).
\]
If $|z|\leq1/2$, use $2|z|^3\leq|z|^2$;
otherwise, use $2|z|^3\leq4|z|^4$.
\end{itemize}

Since $\kappa\leq\phi(z)\leq4$, the chain rule gives
\[
 (H^n)^\#(z)
 \leq\frac{\phi(H^n(z))}{\phi(z)}
 \leq\frac4\kappa=c^{-4}=|D|^{-1}.
\]
This completes the uniform bound on the derivatives of all iterates in Case~3.

\begin{proof}[Proof of Theorem~\ref{sy:cor:normalform}]
The three cases exhaust all possibilities.
Cases~1 and~3 give the uniform bound $|D|^{-1}$, while
Case~2 gives $d^{-4}$, where $d=|D|^{1/3}$.
Each bound is finite and independent of the starting point
and the iterate. Lemma~\ref{tk:lem:fatou} therefore gives
$\JI(H)=\varnothing$.
\end{proof}

\section{The three-adic case}
\label{sec:three-adic}

The goal of this section is to prove the following
uniform derivative bound.

\begin{theorem}
\label{p3:thm:core}
Let $H\in\Cthree(z)$ satisfy
\eqref{tk:eq:H}--\eqref{tk:eq:norms}.
If $H$ has no classical repelling periodic point, then
\[
 (H^n)^\#(z)\leq |D|^{-1}
 \qquad\text{for all }z\in\PP(\Cthree)
 \text{ and }n\geq0.
\]
Consequently, $\JI(H)=\varnothing$.
\end{theorem}

Throughout this section, assume that $H$ satisfies
these hypotheses. Retain $P,Q,\mathcal E$ and put
\[
 d=|D|^{1/3},\qquad
 \delta(z)=
 \begin{cases}
  \min\{|z|,1\},&z\in\Cthree,\\
  1,&z=\infty.
 \end{cases}
\]
Proposition~\ref{tk:prop:coefficientbounds} gives
\[
 |D|\geq|C|^2,\qquad
 |B|\leq\max\{|C|^{1/2},d\}.
\]

We first obtain restrictions on the coefficients and zeros
of $P$ that control $H'$ on $|z|=d$.
We then combine this estimate with the formulas on
$0<|z|<d$ and $d<|z|<1$ to prove the uniform derivative
bound using $\phi(z)=\max\{|D|,\delta(z)\}$.

\subsection{Initial coefficient restrictions}
\label{p3:sec:generic}
\begin{proposition}
\label{p3:prop:generic}
Only Case~2 is possible. Every zero of $P$ has absolute
value $d$, and any two zeros $\zeta,\eta$ satisfy
$|\zeta-\eta|<d$. Moreover,
\[
 |B|<d,\qquad |C|<d^2,\qquad |D|>|C|^{3/2}.
\]
\end{proposition}

\begin{proof}
To exclude Cases~1 and~3, suppose that
$|C|^2\leq|D|<|C|^{3/2}$. Put
$c=|C|^{1/2}$ and $\lambda=|D|/|C|$. Then
$c^2\leq\lambda<c$ and $|B|\leq c$.
The Newton polygon gives one zero $\gamma_0$ of $P$ with
absolute value $\lambda$ and two zeros $\zeta_1,\zeta_2$
with absolute value $c$.
Write
\[
 g(z)=(z-\zeta_1)(z-\zeta_2),\qquad
 H(z)=\frac{A(z-\gamma_0)}{zQ(z)}g(z).
\]
On $|z|=c$, the factor $A(z-\gamma_0)/(zQ(z))$
has norm one, and its derivative has norm at most $\lambda/c^2$.

For $|y|<c^2$, write $z=\zeta_1u$.
The solutions of $H(z)=y$ with $|z|=c$ reduce to solutions of
\[
 \overline{A\zeta_1^2/C}\,\bar u^2
 +\overline{B\zeta_1/C}\,\bar u+1=0.
\]
If $\DD(\zeta_1,c)$ and $\DD(\zeta_2,c)$ are disjoint,
$H$ maps each disk onto $\DD(0,c^2)$ with mapping degree one.
If the disks coincide, $H$ maps their common disk onto
$\DD(0,c^2)$ with mapping degree two.
These alternatives correspond to two simple roots or one
double root in the reduced quadratic equation.
Choose $\DD(\zeta_1,c)$ and a solution $\xi$ of
$H(\xi)=\zeta_1$ with $|\xi|=c\lambda$.
Lemma~\ref{tk:lem:poledisk} gives
\[
 \DD(\xi,c\lambda)
 \xrightarrow{H}\DD(\zeta_1,c)
 \xrightarrow{H}\DD(0,c^2).
\]
Since
$\sup_{z\in\DD(\xi,c\lambda)}|z|=c\lambda<c^2$,
Lemma~\ref{tk:lem:count} gives
$h\in\DD(\xi,c\lambda)$ with $H^2(h)=h$.
For $x=H(h)$, we have
\[
 |h|=c\lambda,\qquad |x|=c,\qquad |H'(h)|=1/\lambda.
\]
Thus $|H'(x)|\leq\lambda$. Since $|g(x)|=c\lambda$,
the derivative of $A(z-\gamma_0)/(zQ(z))$ at $x$,
multiplied by $g(x)$, has norm at most
$\lambda^2/c<\lambda$. Hence $|g'(x)|\leq\lambda$.

If $|x-\zeta_1|\ne|x-\zeta_2|$, then
$g'(x)=(x-\zeta_1)+(x-\zeta_2)$ would have norm
greater than $\lambda$.
Hence the distances from $x$ to both roots equal $(c\lambda)^{1/2}$.
Since $|2|=1$, the identity
$g'(x)=2(x-\zeta_1)+(\zeta_1-\zeta_2)$ gives
\[
 |\zeta_1-\zeta_2|=(c\lambda)^{1/2}<c.
\]
Thus $\DD(\zeta_1,c)=\DD(\zeta_2,c)$.
Choose a disk whose points have absolute value $c^{1/2}$
and which $H$ maps bijectively onto $\DD(\zeta_1,c)$.
Then choose a disk near zero that $H$ maps bijectively
onto that disk. This gives
\[
 \DD(x_1,c^{3/2}\lambda)\xrightarrow{H}\DD(x_0,c^{1/2})
 \xrightarrow{H}\DD(\zeta_1,c)\xrightarrow{H}\DD(0,c^2).
\]
The first two maps are bijective and the last has degree two.
Since $|x_1|=c^{3/2}\lambda<c^2$,
Lemma~\ref{tk:lem:count} gives
$h\in\DD(x_1,c^{3/2}\lambda)$ with $H^3(h)=h$.
At $v=H^2(h)$,
\[
 |g(v)|=c^{3/2}\lambda<|\zeta_1-\zeta_2|^2.
\]
One of $|v-\zeta_1|$ and $|v-\zeta_2|$ is therefore
less than $|\zeta_1-\zeta_2|$, while the other equals it.
Hence $|g'(v)|=(c\lambda)^{1/2}$.
The derivative of $A(z-\gamma_0)/(zQ(z))$ at $v$,
multiplied by $g(v)$, has norm at most
$c^{-1/2}\lambda^2<(c\lambda)^{1/2}$.
Thus $|H'(v)|=(c\lambda)^{1/2}$.
The multiplier of $H^3$ at $h$ has norm
\[
 |(H^3)'(h)|=(c\lambda)^{-1}c^{1/2}(c\lambda)^{1/2}
 =\lambda^{-1/2}>1,
\]
a contradiction.

In Case~2, $|B|\leq d$ and $|C|\leq d^2$, so the
Newton polygon gives three zeros of $P$ with absolute value $d$.
If two of these zeros are at distance $d$,
one zero $\zeta$ is at distance $d$ from each of the
other two. Then $|P'(\zeta)|=d^2$.
The higher Taylor coefficients of $H$ at $\zeta$ are
bounded by $d^{2-k}$ in degree $k\geq2$, so its linear
term dominates on $\DD(\zeta,d)$. This disk maps
bijectively onto $\DD(0,d^2)$ with derivative norm $d$.
Choose a disk whose points have absolute value $d^{1/2}$
and which $H$ maps bijectively onto $\DD(\zeta,d)$.
Then choose a disk near zero that maps bijectively onto that
disk. This gives
\[
 \DD(x_1,d^{5/2})\xrightarrow{H}\DD(x_0,d^{1/2})
 \xrightarrow{H}\DD(\zeta,d)\xrightarrow{H}\DD(0,d^2).
\]
Every point in its first disk has absolute value $d^{5/2}<d^2$.
Lemma~\ref{tk:lem:count} gives a fixed point $h$ of $H^3$
in that disk, with multiplier norm
$d^{-2}d^{1/2}d=d^{-1/2}>1$,
again a contradiction. Thus the distance between any two
zeros of $P$ is less than $d$.
For any zero $\zeta$, the other two zeros of $P(\zeta+y)$
have absolute value less than $d$.
Their sum and product give
\[
 |B+3A\zeta|<d,\qquad |P'(\zeta)|<d^2.
\]
Since
\[
 P'(\zeta)-2\zeta(B+3A\zeta)=C-3A\zeta^2,
\]
we obtain $|C-3A\zeta^2|<d^2$.
Since $|3|<1$, we obtain $|B|<d$ and $|C|<d^2$.
\end{proof}

\subsection{Further coefficient restrictions}
\label{p3:sec:core}

To control $H'$ on $|z|=d$, we need bounds on the coefficients
in the expansion of $P$ at a zero.
We obtain these bounds by constructing periodic points
and using the assumption that none is repelling.

Choose a zero $\zeta$ of $P$, and let $R$ be the largest
distance between two zeros of $P$.
Proposition~\ref{p3:prop:generic} gives $R<d$. Expanding $P$ at $\zeta$ gives
\begin{equation}
\label{p3:eq:root}
 P(\zeta+y)=Ay^3+(3A\zeta+B)y^2+P'(\zeta)y.
\end{equation}
The factorization into linear factors gives
\[
 R=\max\{|3A\zeta+B|,|P'(\zeta)|^{1/2}\}.
\]
Consequently, the cubic term dominates whenever $|y|>R$.
If $R=0$, then $P(z)=A(z-\zeta)^3$.

The next lemma supplies the disk maps used to construct
these periodic points.
Starting with a disk centered at a point of absolute value $d$,
we repeatedly choose preimage disks whose points satisfy
$d<|z|<1$. We then choose a disk near zero that maps
bijectively onto the last chosen disk.
The lemma gives the radius of the first disk and the
derivative norm of the resulting iterate.

\begin{lemma}
\label{p3:lem:inversechain}
For $d<|z|<1$, we have
\begin{equation}
\label{p3:eq:outer}
 |H(z)|=|z|^2,\qquad |H'(z)|=|z|.
\end{equation}
For $0<|z|<d$, we have
\begin{equation}
\label{p3:eq:pole}
 |H(z)|=d^3/|z|,\qquad |H'(z)|=d^3/|z|^2.
\end{equation}
For $k\geq1$, put $r=d^{1/2^k}$. If $|v|=d$ and
$0<T\leq d$, there is a bijective map
\[
 H^{k+1}:\DD(x_1,d^2T/r)\longrightarrow\DD(v,T),
 \qquad |x_1|=d^3/r,
\]
with constant derivative norm
\[
 |(H^{k+1})'(z)|=r/d^2
 \qquad\text{for all }z\in\DD(x_1,d^2T/r).
\]
None of the disks used in this composition contains a pole of $H$.
\end{lemma}

\begin{proof}
For $d<|z|<1$, the term $Az^2$ determines $|H(z)|$.
For $0<|z|<d$, the term $D/z$ determines $|H(z)|$.
Differentiation gives the stated derivative norms;
the first formula uses $|2|=1$.

For $d<s<1$, let $|a|=s^2$ and $0<\tau\leq s^2$.
Choose a solution $b$ of $H(b)=a$ with $|b|=s$.
Writing $z=bu$, the reduction of $H(bu)/a$ is $\bar u^2$,
so the reduced equation for $H(z)=a$ is
\[
 \bar u^2=1.
\]
Its roots $1$ and $-1$ are simple in characteristic three.
Thus
\[
 H:\DD(b,\tau/s)\longrightarrow\DD(a,\tau)
\]
is bijective, with $|H'(z)|=s$ for every
$z\in\DD(b,\tau/s)$.
The other solution of $H(z)=a$ with absolute value $s$
gives a second such disk.

For a target $\DD(a,\tau)$ with $|a|=s>d$ and $0<\tau\leq s$,
choose the solution $b$ of $H(b)=a$ with $|b|=d^3/s$.
Lemma~\ref{tk:lem:poledisk} and restriction to the smaller
target give a bijective map
\[
 H:\DD(b,\tau d^3/s^2)\longrightarrow\DD(a,\tau),
\]
with $|H'(z)|=s^2/d^3$ for every
$z\in\DD(b,\tau d^3/s^2)$.
These assertions follow by clearing denominators in
$H(z)=y$, which gives
\[
 (A-yF)z^3+(B-yE)z^2+(C-y)z+D=0,
\]
and examining the linear term in the expansion of $H$
on each chosen disk.

The points in the successive disks of these $k$ steps have
absolute values $d^{1/2},d^{1/4},\ldots,r$, whose product is $d/r$.
These steps pull back the target to a disk of radius $rT/d$.
Finally, the disk near zero mapping bijectively onto the
last chosen disk has radius $d^2T/r$.
The product of the derivative norms is then
$(r^2/d^3)(d/r)=r/d^2$.
\end{proof}

\begin{lemma}
\label{p3:lem:coefficientbounds}
For the chosen zero $\zeta$, we have
\begin{equation}
\label{p3:eq:criterion}
 d\geq\frac1{27},\qquad
 |3A\zeta+B|\leq d^{5/3},\qquad |P'(\zeta)|\leq d^3.
\end{equation}
\end{lemma}

\begin{proof}
First we show that $R\leq d^{4/3}$. If $R>d^{4/3}$,
one zero $v$ is at distance $R$ from both other zeros.
Thus $|P'(v)|=R^2$.
The coefficient of $y^2$ in $P(v+y)$ satisfies
$|3Av+B|\leq R$, so the linear term dominates on
$\DD(v,R)$. Hence
\[
 H:\DD(v,R)\longrightarrow\DD(0,R^3/d)
\]
is bijective, with derivative norm $R^2/d$.

Choose $k$ so large that $r=d^{1/2^k}>d^4/R^3$.
Lemma~\ref{p3:lem:inversechain}, with $T=R$, gives a
return whose source disk consists entirely of points of
absolute value $d^3/r$, which is strictly smaller than $R^3/d$.
Lemma~\ref{tk:lem:count} gives a fixed point $h$ of $H^{k+2}$.
The multiplier of $H^{k+2}$ at $h$ has norm
\[
 |(H^{k+2})'(h)|=(r/d^2)(R^2/d)=rR^2/d^3>d/R>1,
\]
a contradiction. Thus $R\leq d^{4/3}$.

To prove \eqref{p3:eq:criterion}, take $k\geq1$, put
$r=d^{1/2^k}$, and choose a radius
\[
 d^{4/3}r^{-1/3}<T<d.
\]
Since $T>R$, the cubic term gives a degree-three map
$\DD(\zeta,T)\to\DD(0,T^3/d)$.
Lemma~\ref{p3:lem:inversechain} gives the composition
\[
 \DD(x_1,d^2T/r)
 \xrightarrow{H^{k+1}}\DD(\zeta,T)
 \xrightarrow{H}\DD(0,T^3/d),
 \qquad |x_1|=d^3/r<T^3/d.
\]
Lemma~\ref{tk:lem:count} gives
$h\in\DD(x_1,d^2T/r)$ with $H^{k+2}(h)=h$.
Put $z=H^{k+1}(h)$. Then
\[
 |z|=d,\qquad H(z)=h,\qquad |P(z)|=d^4/r>d^4.
\]
$|P(z)|$ would be at most $R^3\leq d^4$ if
$|z-\zeta|\leq R$. Therefore the cubic term dominates,
and
\[
 |z-\zeta|=d^{4/3}r^{-1/3}.
\]

Differentiating \eqref{p3:eq:root} gives
\begin{equation}
\label{p3:eq:terminalder}
 H'(z)=
 \frac{3A(z-\zeta)^2+2(3A\zeta+B)(z-\zeta)+P'(\zeta)}
      {zQ(z)}
 -H(z)\left(\frac1z+\frac{Q'(z)}{Q(z)}\right).
\end{equation}
On $|z|=d$, the expression $1/z+Q'(z)/Q(z)$ has norm $1/d$.
Consider separately the three terms in the numerator of
\eqref{p3:eq:terminalder} and the final term
$-H(z)(1/z+Q'(z)/Q(z))$.
Multiplying their norms in $H'(z)$ by the derivative norm
of $H^{k+1}$, namely $r/d^2$, gives
\begin{equation}
\label{p3:eq:fourterms}
 \frac{r^{1/3}}{3d^{1/3}},\qquad
 \frac{|3A\zeta+B|r^{2/3}}{d^{5/3}},\qquad
 \frac{|P'(\zeta)|r}{d^3},\qquad 1.
\end{equation}
If one of the four numbers in \eqref{p3:eq:fourterms} is
strictly larger than the others, the ultrametric inequality
shows that it equals $|(H^{k+2})'(h)|$.
As $k\to\infty$, $r=d^{1/2^k}$ tends to one from below.
Suppose that the limit of one of the first three numbers
exceeds one. For all sufficiently large $k$, numbers with
smaller limits are smaller than those with the largest limit.
If several limits are equal, their coefficients are equal,
and the smallest exponent among $1/3,2/3,1$ gives the largest
value for $0<r<1$.
Thus, for all sufficiently large $k$, there is a unique
largest number, and it is greater than one.
The corresponding fixed point $h$ of $H^{k+2}$ is then
repelling, a contradiction.
All three bounds in \eqref{p3:eq:criterion} follow.
\end{proof}

\subsection{Construction of \texorpdfstring{$\phi$}{phi} and the uniform derivative bound}

The coefficient bounds in Lemma~\ref{p3:lem:coefficientbounds}
give the derivative estimate on $|z|=d$ needed to construct
$\phi$.
\begin{lemma}
\label{p3:lem:circlederivative}
For $|z|=d$, we have
\begin{equation}
\label{p3:eq:terminalweight}
 |H'(z)|\leq d^{-1}\max\{|H(z)|,d^3\}.
\end{equation}
\end{lemma}

\begin{proof}
By \eqref{p3:eq:root} and
Lemma~\ref{p3:lem:coefficientbounds}, if
$|z-\zeta|\leq d^{4/3}$, then
\[
 |P(z)|\leq d^4,\qquad |P'(z)|\leq d^3,
\]
where the second inequality uses $d\geq1/27$.
If $|z-\zeta|>d^{4/3}$, the cubic term gives
$|P(z)|=|z-\zeta|^3$, and each term of $P'(z)$ has
norm at most $|P(z)|/d$. For the term $3A(z-\zeta)^2$,
use $|z-\zeta|>d^{4/3}\geq d/3$; the other two terms
are bounded by \eqref{p3:eq:criterion}.
Dividing by $zQ(z)$ and including the term
$-H(z)(1/z+Q'(z)/Q(z))$
in \eqref{p3:eq:terminalder} proves \eqref{p3:eq:terminalweight}.
\end{proof}

We are now ready to construct $\phi$.
The formulas in Lemma~\ref{p3:lem:inversechain} give the
required comparison on $0<|z|<d$ and $d<|z|<1$.
Lemma~\ref{p3:lem:circlederivative} gives the comparison
on $|z|=d$.

\begin{proof}[Proof of Theorem~\ref{p3:thm:core}]
Define
\[
 \phi(z)=\max\{d^3,\delta(z)\}.
\]
Then $d^3\leq\phi(z)\leq1$. We verify
\begin{equation}
\label{p3:eq:one-step}
 H^\#(z)\phi(z)\leq\phi(H(z)).
\end{equation}

On $\mathcal E$, forward invariance gives
$\phi(z)=\phi(H(z))=1$. Lemma~\ref{tk:lem:derivativeestimates}
gives $H^\#\leq1$ there.

If $|z|\leq d^3$, the formula near the pole gives
$H(z)\in\mathcal E$ and $H^\#(z)=d^{-3}$, so equality
holds. At zero, use $1/H(z)=zQ(z)/P(z)$.

If $d^3<|z|<d$, both $z$ and $H(z)$ lie in the open unit disk,
and \eqref{p3:eq:pole} gives
\[
 H^\#(z)\phi(z)=\frac{d^3}{|z|^2}|z|
 =|H(z)|=\phi(H(z)).
\]

On $|z|=d$, we have $|H(z)|\leq d^2<1$.
Lemma~\ref{p3:lem:circlederivative} gives
\[
 H^\#(z)\phi(z)=d|H'(z)|
 \leq\max\{|H(z)|,d^3\}=\phi(H(z)).
\]

Finally, if $d<|z|<1$, \eqref{p3:eq:outer} gives
\[
 H^\#(z)\phi(z)=|z|^2=\phi(H(z)).
\]
The chain rule and \eqref{p3:eq:one-step} give
\[
 (H^n)^\#(z)
 \leq\frac{\phi(H^n(z))}{\phi(z)}
 \leq d^{-3}=|D|^{-1}
 \qquad\text{for all }z\in\PP(\Cthree)
 \text{ and }n\geq0.
\]
Lemma~\ref{tk:lem:fatou} therefore gives
$\JI(H)=\varnothing$.
\end{proof}

\begin{proof}[Proof of Theorem~\ref{thm:main-global}]
Suppose that $f$ has no classical repelling periodic point.
For $p\geq5$, Proposition~\ref{prop:tame} gives
$\JI(f)=\varnothing$.
For $p=2$ or $p=3$, a map with potential good reduction
already has empty classical Julia set. Otherwise,
Theorem~\ref{nm:thm:normalization} gives a conjugate $H$
satisfying \eqref{tk:eq:H}--\eqref{tk:eq:norms}.
Theorem~\ref{sy:cor:normalform} for $p=2$ and
Theorem~\ref{p3:thm:core} for $p=3$ give
$\JI(H)=\varnothing$. Conjugation therefore gives
$\JI(f)=\varnothing$.

Thus a nonempty classical Julia set contains a classical
repelling periodic point. In this case, B\'ezivin's theorem
\cite[Theorem~3]{bezivin} gives
\[
 \JI(f)=\overline{\Rep(f)}.
\]
If $\JI(f)=\varnothing$, the same equality holds because
every classical repelling periodic point belongs to
$\JI(f)$.
\end{proof}
\section*{Disclosure of AI assistance}

We used OpenAI's ChatGPT to assist with language
editing, organization, checking calculations, and reviewing
the argument. We reviewed all AI-assisted revisions
and take full responsibility for the contents of the paper.

\end{document}